\documentclass{scrartcl}
\usepackage{amsmath}
\usepackage{amsthm}
\usepackage{amsfonts}
\usepackage{amssymb}
\usepackage{amscd}
\usepackage{mathrsfs}

\numberwithin{equation}{section} 

\let\C\relax
\newcommand{\coloneqq}{:=}
\newenvironment{pdeq}{ \left\{ \begin{aligned}}{\end{aligned}\right.}

\newcommand{\np}[1]{(#1)}
\newcommand{\nb}[1]{[#1]}
\newcommand{\bp}[1]{\big(#1\big)}
\newcommand{\bb}[1]{\big[#1\big]}
\newcommand{\Bp}[1]{\bigg(#1\bigg)}

\newcommand{\Bb}[1]{\bigg[#1\bigg]}
\newcommand{\cala}{{\mathcal A}}
\newcommand{\calb}{{\mathcal B}}

\newcommand{\calf}{{\mathcal F}}

\newcommand{\calh}{{\mathcal H}}

\newcommand{\calm}{{\mathcal M}}
\newcommand{\caln}{{\mathcal N}}

\newcommand{\calp}{{\mathcal P}}

\newcommand{\calt}{{\mathcal T}}

\newcommand{\R}{\mathbb{R}}
\newcommand{\C}{\mathbb{C}}
\newcommand{\Z}{\mathbb{Z}}
\newcommand{\N}{\mathbb{N}}

\DeclareMathOperator{\e}{e}

\DeclareMathOperator{\Div}{div}

\DeclareMathOperator{\supp}{supp}
\DeclareMathOperator*{\esssup}{ess\,sup}

\DeclareMathOperator{\curl}{curl}

\DeclareMathOperator{\impart}{Im}

\newcommand{\set}[1]{\ensuremath{\{#1\}}}

\newcommand{\setc}[2]{\ensuremath{\{#1\ \vert\ #2\}}}
\newcommand{\setcl}[2]{\ensuremath{\bigl\{#1\ \big\vert\ #2\bigr\}}}

\newcommand{\ball}{\mathrm{B}}

\renewcommand{\restriction}[2]{#1 | _{#2}}
\newcommand{\proj}{\calp}
\newcommand{\projcompl}{\calp_\bot}

\newcommand{\grp}{G}
\newcommand{\dualgrp}{\widehat{G}}

\newcommand{\torus}{{\mathbb T}}

\newcommand{\idmatrix}{I}
\newcommand{\grad}{\nabla}

\newcommand{\D}{{\mathrm D}}
\newcommand{\dx}{{\mathrm d}x}

\newcommand{\dtx}{{\mathrm d}(t,x)}
\newcommand{\dsy}{{\mathrm d}(s,y)}

\newcommand{\dtau}{{\mathrm d}\tau}

\newcommand{\ds}{{\mathrm d}s}
\newcommand{\dt}{{\mathrm d}t}

\newcommand{\dy}{{\mathrm d}y}

\newcommand{\dS}{{\mathrm d}S}
\newcommand{\dxi}{{\mathrm d}\xi}

\newcommand{\SR}{\mathscr{S}}

\newcommand{\TDR}{\mathscr{S^\prime}}

\newcommand{\FT}{\mathscr{F}}
\newcommand{\iFT}{\mathscr{F}^{-1}}

\newcommand{\mmultiplier}{m}
\newcommand{\Mmultiplier}{M}

\newcommand{\norm}[1]{\lVert#1\rVert}
\newcommand{\norml}[1]{\big\lVert#1\big\rVert}

\newcommand{\snorm}[1]{{\lvert #1 \rvert}}
\newcommand{\snorml}[1]{{\bigl\lvert #1 \big\rvert}}
\newcommand{\snormL}[1]{{\Bigl\lvert #1 \Big\rvert}}

\newcommand{\velnorm}[1]{\mathrm{M}_\radius\np{#1}}
\newcommand{\vortnorm}[1]{\mathrm{N}_\radius^\varepsilon\np{#1}}
\newcommand{\LR}[1]{\mathrm{L}^{#1}}

\newcommand{\LRloc}[1]{\mathrm{L}^{#1}_{\mathrm{loc}}} 
\newcommand{\CR}[1]{\mathrm{C}^{#1}}  
\newcommand{\CRi}{\CR \infty}
\newcommand{\CRci}{\CR \infty_0}

\newcommand{\WSR}[2]{\mathrm{W}^{#1,#2}} 
 
\newcommand{\DSR}[2]{\mathrm{D}^{#1,#2}} 
\newcommand{\DSRN}[2]{\mathrm{D}^{#1,#2}_0}

\newcommand{\WSRloc}[2]{\mathrm{W}^{#1,#2}_{\mathrm{loc}}}

\newcommand{\velspace}{\calm_\radius}
\newcommand{\vortspace}{\caln_\radius^\varepsilon}

\newcommand{\DSRNsigma}[2]{\mathrm{D}^{#1,#2}_{0,\sigma}}

\newcommand{\CRcisigma}{\CR{\infty}_{0,\sigma}}

\newcommand{\vvel}{v}
\newcommand{\vpres}{p}

\newcommand{\Vvel}{V}

\newcommand{\wvel}{w}
\newcommand{\wpres}{\mathfrak{q}}

\newcommand{\zvel}{z}
\newcommand{\zvelss}{z_0}
\newcommand{\zvelpp}{z_\perp}

\newcommand{\uvel}{u}

\newcommand{\upres}{\mathfrak{p}}

\newcommand{\Uvel}{U}

\newcommand{\Upres}{\mathfrak{P}}
\newcommand{\tuvel}{\tilde{u}}

\newcommand{\wakefct}[1]{s(#1)}
\newcommand{\fundsolvel}{\varGamma^\rey}

\newcommand{\fundsolvelss}{\fundsolvel_0}
\newcommand{\fundsolvelssjl}{\fundsolvel_{0,j\ell}}

\newcommand{\fundsolvelpp}{\fundsolvel_\perp}
\newcommand{\fundsolvelppjl}{\fundsolvel_{\perp,j\ell}}

\newcommand{\fundsolvort}{\phi^\rey}
\newcommand{\fundsolvortss}{\fundsolvort_0}
\newcommand{\fundsolvortpp}{\fundsolvort_\perp}

\newcommand{\fundsolaux}{\varGamma_{\mathrm{H}}^{\perf k,\rey}}
\newcommand{\fundsolauxeta}{\varGamma_{\mathrm{H}}^{\eta,\rey}}
\newcommand{\fundsolauxetaperf}{\varGamma_{\mathrm{H}}^{\perf\eta,\rey}}

\newcommand{\tin}{\text{in }}

\newcommand{\ton}{\text{on }}

\newcommand{\tfor}{\text{for }}

\newcommand{\half}{\frac{1}{2}}

\newcommand{\fourth}{\frac{1}{4}}

\renewcommand{\epsilon}{\varepsilon}
\newcommand{\rey}{\lambda}

\newcommand{\per}{\calt}

\newcommand{\perf}{\frac{2\pi}{\per}}

\newcommand{\eone}{\e_1}

\newcommand{\ej}{\e_j}

\newcommand{\cutoff}{\chi}

\newcommand{\onedist}{1}
\newcommand{\change}[1]{}

\newcommand{\radius}{S}
\newcommand{\Kconst}{K}
\newcommand{\nonlin}{\cala}
\newcommand{\nonlinss}{\cala_0}
\newcommand{\nonlinpp}{\cala_\perp}
\newcommand{\nvec}{\mathrm{n}}

\newcommand{\Cexp}{C_4}
\newcommand{\Cexpfs}{C_3}
\newcommand{\CexpfsMult}{C_5}
\newcommand{\CVelocityDecay}{C_2}
\newcommand{\CVorticityDecay}{C_1}
\newcommand{\CHsvelnorm}{C_6}
\newcommand{\CHsvortnorm}{C_7}
\newcommand{\CFsvelnorm}{C_8}
\newcommand{\CFsvortnorm}{C_9}

\theoremstyle{plain}
\newtheorem{thm}{Theorem}[section]
\newtheorem{defn}[thm]{Definition}
\newtheorem{lem}[thm]{Lemma}
\newtheorem{prop}[thm]{Proposition}
\newtheorem{cor}[thm]{Corollary}
\theoremstyle{remark}
\newtheorem{rem}[thm]{Remark}

\begin{document}
\title{Spatial decay of the vorticity field of time-periodic viscous flow past a body}

\author{Thomas Eiter%
\thanks{Partially supported by Deutsche Forschungsgemeinschaft 
(DFG, German Research Foundation), 
project number: 427538878} 
\and 
Giovanni P. Galdi}

\maketitle

\begin{abstract}
We study the asymptotic spatial behavior of the vorticity field, $\omega(x,t)$, associated to a 
time-periodic Navier--Stokes flow past a body, $\mathscr B$, in the class of weak solutions satisfying a Serrin-like condition. 
We  show that, outside the wake region, $\mathcal R$,
$\omega$ decays pointwise at an  exponential rate, uniformly in time. Moreover, denoting by   $\bar{\omega}$  its time-average over a period and by $\omega_P:=\omega-\bar{\omega}$ its purely periodic component, we prove that inside $\mathcal R$, $\bar{\omega}$ has the same algebraic decay 
as that known for the associated steady-state problem, whereas $\omega_P$ decays even faster, uniformly in time. 
This implies, in particular, that ``sufficiently far" from $\mathscr B$, $\omega(x,t)$ behaves like the vorticity field of the corresponding steady-state problem.
\end{abstract}


\section{Introduction}

Consider a (rigid) body, $\mathscr B$,  translating with constant nonzero  velocity, $\vvel_\infty$, in a viscous (Navier-Stokes) liquid, $\mathscr L$, that occupies the whole space outside $\mathscr B$. Without loss of generality, we  assume that $\vvel_\infty$ is directed along the positive $x_1$-axis, namely,
$\vvel_\infty=\rey\eone$ with $\rey>0$. We also assume that $\mathscr L$ is subject to a body force  and a distribution of boundary velocity, both being time-periodic of period $\mathcal T$. 
Then, the time-periodic dynamics of the liquid around the body 
are governed by the
following set of equations
\begin{equation}\label{sys:NavierStokesTP_ExteriorDomain}
\begin{pdeq}
\partial_t\uvel-\Delta\uvel-\rey\partial_1\uvel+\uvel\cdot\grad\uvel+\grad\upres&=f
&&\tin\torus\times\Omega, \\
\Div\uvel&=0 
&&\tin\torus\times\Omega, \\
\uvel&=\uvel_\ast 
&&\ton\torus\times\partial\Omega, \\
\lim_{\snorm{x}\to\infty}\uvel(t,x)&=0 
&&\tfor t\in\torus,
\end{pdeq}
\end{equation}
where $\Omega\coloneqq\R^3\setminus\mathscr B$ is the domain occupied by the liquid.
Moreover,  $\uvel\colon\torus\times\Omega\to\R^3$ and $\upres\colon\torus\times\Omega\to\R$ are velocity and pressure fields 
of the liquid,
$f\colon\torus\times\Omega\to\R^3$ is the external body force,
and $\uvel_\ast\colon\torus\times\partial\Omega\to\R^3$
the velocity field at the boundary.
The time-axis is given by the torus group 
$\torus\coloneqq\R/\per\Z$, which 
ensures that all functions appearing in \eqref{sys:NavierStokesTP_ExteriorDomain} 
are time-periodic with a prescribed period $\per>0$. 
Note that for
a body at rest, that is, for $\rey=0$, the mathematical and physical characteristics of the flow 
are very different from those for $\rey\neq0$. For this issue, we refer the reader to the recent papers \cite{galdi:zeroaverage,galdi:transarbitrary}.
\par
Existence, uniqueness and spatial asymptotic behavior of solutions  to \eqref{sys:NavierStokesTP_ExteriorDomain} have been the object of several recent researches  \cite{GaldiKyed_TPflowViscLiquidpBody,GaldiKyed_TPSolNS3D_AsymptoticProfile,Eiter_SpatiallyAsymptoticStructureTPNS_2020}. In particular, under suitable assumptions on the data, these results  provide sharp pointwise {\em algebraic} decays for the velocity field and its first spatial derivatives; see \cite{Eiter_SpatiallyAsymptoticStructureTPNS_2020} and Theorem \ref{thm:VelocityDecay} below. However, as  suggested by physical grounds, the vorticity field $\omega:=\curl u$  is expected to decay at an {\em exponential} rate, at least outside the ``wake region" behind $\mathscr B$. It is just to this question that the present paper is devoted.  
\par
More precisely, we shall study  the  asymptotic behavior of 
the  vorticity field $\curl\uvel(x,t)$ for $\snorm{x}\to\infty$, uniformly in time.
In these regards, we recall that in the case of a steady-state flow, that is, when $\np{\vvel,\vpres}$ is a time-independent
solution to \eqref{sys:NavierStokesTP_ExteriorDomain},
a famous result of \textsc{Clark} \cite{Clark_VorticityAtInfinityNavierStokes_1970}
and \textsc{Babenko} and \textsc{Vasil'ev} \cite{BabenkoVasilev_AsymptoticBehaviorSteadyFlow_1973}
shows that for $\snorm{x}$ sufficiently large one has
\begin{equation}\label{est:Curlu_SteadyStateCase}
\snorm{\curl\vvel(x)}\leq C \snorm{x}^{-3/2}\e^{-\alpha \wakefct{x}}
\end{equation}
for some constants $C,\,\alpha>0$,
where 
\[
\wakefct{x}\coloneqq\snorm{x}+x_1.
\]
In particular, this reflects the anisotropic behavior of the fluid flow and
translates, in mathematical terms, the presence of a ``wake region" behind $\mathscr B$. 
Estimate \eqref{est:Curlu_SteadyStateCase} 
implies that the vorticity, $\curl\vvel$, decays exponentially fast on rays 
$\setc{x\in\R^3}{x_1=\theta\snorm{x}}$ for $\theta\in(-1,1]$,
while inside parabolic regions $\setc{x\in\R^3}{\wakefct{x}\leq \beta}$, $\beta>0$,
estimate \eqref{est:Curlu_SteadyStateCase} merely yields an algebraic decay rate.
Since time-independent solutions are trivially also time-periodic, 
one would expect a similar behavior in the time-periodic case.
As a matter of fact, we show that this is indeed true and that 
the vorticity field associated to a time-periodic flow is subject to an analogus estimate.

Actually, as proved  in \cite{Eiter_SpatiallyAsymptoticStructureTPNS_2020},
 if we split $\uvel$ into 
its time average $\vvel$ and a purely periodic part $\wvel:=\uvel-\vvel$,
then the decay rates of $\vvel$ and $\grad\vvel$ are much slower than those of $\wvel$ and $\grad\wvel$.
Thus, also in the problem at hand, it seems reasonable to derive separate pointwise estimates for the two parts
$\curl\vvel$ and $\curl\wvel$ of the vorticity $\curl\uvel$.
In doing so, we are indeed able to show that
the time-independent part $\vvel$ satisfies \eqref{est:Curlu_SteadyStateCase}
whereas the other part obeys the estimate
\begin{equation}\label{est:Curlu_PurelyPeriodic}
\snorm{\curl\wvel(t,x)}\leq C \snorm{x}^{-9/2}\e^{-\alpha \wakefct{x}}
\end{equation}
for all  sufficiently large $\snorm{x}$, 
and therefore decays faster.
It is worth emphasizing that we establish this result for {\em any} weak solution to \eqref{sys:NavierStokesTP_ExteriorDomain} (see Definition \ref{def:WeakSolution_NStp}), whose purely periodic part {\em only} satisfies the Serrin-like condition \eqref{el:VelocityAdditionalIntegrability},  provided the data are sufficiently smooth with $f$ of bounded spatial support; see Theorem \ref{thm:VorticityDecay}.

A main tool in our approach is the introduction of a time-periodic fundamental solution
associated to the vorticity field $\curl\uvel$.
The concept of time-periodic fundamental solutions in the field of fluid dynamics is new
and was recently introduced by \textsc{Kyed} \cite{Kyed_FundsolTPStokes2016}
and \textsc{Galdi} and \textsc{Kyed} \cite{GaldiKyed_TPSolNS3D_AsymptoticProfile}
in the case of a three-dimensional Navier--Stokes flow,
and further extended by \textsc{Eiter} and \textsc{Kyed} \cite{EiterKyed_etpfslns}
to the general $n$-dimensional case.
The fundamental solution $\fundsolvel$ introduced there consists of the fundamental solution $\fundsolvelss$
to the steady-state problem and a second so-called purely periodic part $\fundsolvelpp$. 
Analogously, we define the time-periodic vorticity fundamental solution $\fundsolvort$ 
as the sum 
of the corresponding steady-state fundamental solution $\fundsolvortss$ 
and a purely periodic part $\fundsolvortpp$.

After introducing these time-periodic fundamental solutions,
our procedure parallels that of \cite{DeuringGaldi_ExpDecayVorticity_2016},
where \textsc{Deuring} and \textsc{Galdi} studied the 
vorticity field associated to the steady-state flow past a rotating body.
Note that this problem is directly related to 
the one investigated here 
since a time-independent solution in the frame attached to the rotating body
corresponds to a time-periodic solution in the inertial frame.
By means of the above time-periodic fundamental solutions we deduce representation formulas
for $\uvel$ and $\curl\uvel$,
which enable us to express 
$\uvel$ as a fixed point of a nonlinear map $F_\radius$ of convolution type; see eq. \eqref{eq:uFixedPoint_VortDecay}.
We then 
establish the existence of a fixed point $\zvel=F_\radius(\zvel)$ of this map
in a class of functions such that $\curl\zvel$ decays in the expected way; see Corollary \ref{cor:ExistenceFixedPoint} .
Successively, we show that this fixed point is, in fact, unique in the {\em larger} class of functions
that merely satisfy the pointwise estimates of $\uvel$ and $\grad\uvel$ established in 
\cite{Eiter_SpatiallyAsymptoticStructureTPNS_2020}; see Theorem \ref{thm:VorticityDecay_WholeSpace}.
Since $\uvel$ is a fixed point of $F_\radius$ by construction, 
we thus conclude $\uvel=\zvel$ and that $\uvel=\vvel+\wvel$ satisfies 
\eqref{est:Curlu_SteadyStateCase} and \eqref{est:Curlu_PurelyPeriodic}.
Observe that, in order to employ the contraction mapping principle,
the existence of the fixed point $\zvel$ is established
in a class of functions that satisfy a slightly weaker estimate
than that given in \eqref{est:Curlu_PurelyPeriodic}.
However, by another application of the representation formulas
via the vorticity fundamental solution, we finally
obtain the asserted decay rates \eqref{est:Curlu_SteadyStateCase} and \eqref{est:Curlu_PurelyPeriodic}.
The result just described is proved in the case where $\Omega$ is the whole space $\R^3$. However, we show that it can be readily 
transferred to the case of an exterior domain 
by a classical cut-off argument, provided $u_*$ and $f$ are sufficiently smooth, with $u_*$ having zero total net flux at $\partial\Omega$. We leave it as an open question whether this condition can indeed be removed.

Finally, we observe that some of the intermediate results are contained in the first author's PhD thesis \cite{Eiter_Diss}. However, they were derived
under the stringent assumption that both external force $f$ and solution $\uvel$ are of class $\CRi$. In contrast,  here we merely require summability assumptions on
$f$ and $\uvel$ (see \eqref{cond:forcing} and \eqref{el:VelocityAdditionalIntegrability}) which
represents a rather significant improvement

The paper is structured as follows.
After introducing the basic notation in Section \ref{sec:Notation},
we present our main result on the decay of the vorticity field in Section \ref{sec:MainResults}.
In Section \ref{sec:TPFundSol} we recall the notion of a time-periodic fundamental solution to the Navier--Stokes equations 
and introduce the concept of a time-periodic vorticity fundamental solution.
Section \ref{sec:Regularity} is dedicated to the study of regularity
of weak solutions to the time-periodic Navier--Stokes problem.
The introduced fundamental solutions 
are employed in Section \ref{sec:FixedPointProblem}
in order to conclude a suitable fixed-point equation.
After the derivation of appropriate estimates for the terms in this equation in Section \ref{sec:Estimates},
we finish the proof of the main result in Section \ref{sec:ConclusionProof}.

\section{Notation}
\label{sec:Notation}

Points in $\torus\times\Omega$ for $\Omega\subset\R^3$ are usually denoted by $(t,x)$ 
and consist of a time variable $t\in\torus$
and a spatial variable $x\in\Omega$.
For a sufficiently regular function $\uvel\colon\torus\times\Omega\to\R^3$ 
we write $\partial_j\uvel\coloneqq\partial_{x_j}\uvel$,
and we set $\Delta\uvel\coloneqq\partial_j\partial_j\uvel$ 
and $\Div\uvel\coloneqq\partial_j\uvel_j$.
Here we employ Einstein's summation convention, 
which we do frequently in the following.
By $\delta_{jk}$ and $\varepsilon_{jk\ell}$ we denote
the Kronecker delta and the Levi-Civita symbol, respectively.

For $R>0$ and $x\in\R^3$ we set $\ball_R(x)\coloneqq\setc{y\in\R^3}{\snorm{x-y}<R}$ and
$\ball^R(x)\coloneqq\setc{y\in\R^3}{\snorm{x-y}>R}$,
and in the case $x=0$ we write $\ball_R\coloneqq\ball_R(0)$ and $\ball^R\coloneqq\ball^R(0)$.
Moreover, for $R>r>0$ we set $\ball_{r,R}\coloneqq\ball_R\cap\ball^r$.
For vectors $a,b\in\R^3$ their vector product $a\wedge b$ and their tensor product $a\otimes b$ 
are given by $(a\wedge b)_{j}=\varepsilon_{jk\ell}a_k b_\ell$ and
$(a\otimes b)_{jk}=a_jb_k$, respectively.
Moreover, we call a subset $\Omega\subset\R^3$ an exterior domain, 
if it is the complement of a non-empty compact subset of $\R^3$. 
Without loss of generality, we always assume that $0$ is contained 
in the interior of $\R^3\setminus\Omega$.

In order to include the time periodicity in the formulation
of the Navier--Stokes equations \eqref{sys:NavierStokesTP_ExteriorDomain},
we formulated the system
on $\torus\times\Omega$. 
In the case $\Omega=\R^3$, which plays a prominent role in our approach, 
the time-space domain is given by the locally compact Abelian group $\grp\coloneqq\torus\times\R^3$.
The dual group of $\grp$ can be identified with $\dualgrp=\Z\times\R^3$,
the elements of which we denote by $(k,\xi)\in\Z\times\R^3$.
We equip the group $\torus$ with the normalized Haar measure given by
\[
\forall f\in\CR{}(\torus):\qquad
\int_\torus f(t)\,\dt=\frac{1}{\per}\int_0^\per f(t)\,\dt,
\]
the group $\Z$ with the counting measure, 
and $\grp$ and $\dualgrp$ with the corresponding product measures.
The Fourier transform $\FT_\grp$ on $\grp$
and its inverse $\iFT_\grp$ are formally given by
\[
\begin{aligned}
\FT_\grp\nb{f}(k,\xi)
&\coloneqq\int_\torus\int_{\R^3} f(t,x)\e^{-i\perf kt-ix\cdot\xi}\,\dx\dt,\\
\iFT_\grp\nb{f}(t,x)
&\coloneqq\sum_{k\in\Z}\int_{\R^3} f(k,\xi)\e^{i\perf kt+ix\cdot\xi}\,\dxi.\\
\end{aligned}
\]
This defines an isomorphism $\FT_\grp\colon\SR(\grp)\to\SR(\dualgrp)$ with inverse $\iFT_\grp$,
provided that the Lebesgue measure $\dxi$ is normalized appropriately.
Here $\SR(\grp)$ is the so-called Schwartz--Bruhat space, which is a generalization 
of the classical Schwartz space in the Euclidean setting;
see \cite{Bruhat61,EiterKyed_tplinNS_PiFbook}.
By duality, this yields an isomorphism $\FT_\grp\colon\TDR(\grp)\to\TDR(\dualgrp)$
between the corresponding dual spaces $\TDR(\grp)$ and $\TDR(\dualgrp)$,
the spaces of tempered distributions.

For an open set $\Omega\subset\R^3$ or $\Omega\subset\torus\times\R^3$ and $q\in[1,\infty]$, $m\in\N$,
we denote the classical Lebesgue and Sobolev spaces  by $\LR{q}(\Omega)$
and $\WSR{m}{q}(\Omega)$, respectively.
Moreover, $\LRloc{1}(\Omega)$ is the set of all locally integrable functions,
and $\WSRloc{1}{1}(\Omega)$ is the subset of $\LRloc{1}(\Omega)$ with 
locally integrable weak derivatives.
For an open subset $\Omega\subset\R^3$,
homogeneous Sobolev spaces are denoted by 
\[
\DSR{m}{q}(\Omega)
\coloneqq\setcl{\uvel\in\LRloc{1}(\Omega)}{\grad^m \uvel\in\LR{q}(\Omega)}, 
\] 
where $\grad^m\uvel$ denotes the collection of all $m$-th 
weak derivatives of $\uvel$.
We further set
\[
\CRcisigma(\Omega)\coloneqq\setc{\varphi\in\CRci(\Omega)^3}{\Div\varphi=0},
\]
where $\CRci(\Omega)$ is the class of 
compactly supported smooth functions on $\Omega$.
For $q\in[1,\infty]$ and a (semi-)normed vector space $X$,
$\LR{q}(\torus;X)$ denotes the corresponding Bochner--Lebesgue space on $\torus$,
and
\[
\WSR{1,2}{q}(\torus\times\Omega)
\coloneqq
\setcl{\uvel\in\LR{q}(\torus;\WSR{2}{q}(\Omega))}
{\partial_t\uvel\in\LR{q}(\torus\times\Omega)}.
\]

We further define the projections
\[
\proj f (x)\coloneqq \int_\torus f(t,x)\,\dt, 
\qquad \projcompl f\coloneqq f-\proj f,
\] 
which decompose $f\in\LRloc{1}\np{\torus\times\Omega}$ 
into a time-independent \emph{steady-state} part $\proj f$
and a remainder \emph{purely periodic} part $\projcompl f$.
One readily sees that $\proj$ and $\projcompl$ are bounded operators on $\LR{q}(\torus\times\Omega)$ for all $q\in[1,\infty]$
and that
\[
\proj f = \iFT_\grp \bb{\delta_\Z(k) \FT_\grp\nb{f}}, 
\qquad
\projcompl f = \iFT_\grp \bb{\np{1-\delta_\Z(k)} \FT_\grp\nb{f}}, 
\]
where $\delta_\Z$ is the delta distribution on $\Z$.

The letter $C$ always denotes a generic positive constant, 
the value of which may change from line to line.
When we want to specify the dependence of the constant $C$ on quantities $a, b, \ldots$, 
we write $C(a, b, \ldots)$.

\section{Main result}
\label{sec:MainResults}
As emphasized earlier on, our focus  
is the pointwise estimates of the vorticity field $\curl \uvel$
associated to a solution $\np{\uvel,\upres}$ of
\eqref{sys:NavierStokesTP_ExteriorDomain}.
More precisely, we study the vorticity field of weak solutions 
to \eqref{sys:NavierStokesTP_ExteriorDomain}
defined as follows.

\begin{defn}\label{def:WeakSolution_NStp}
Let $f\in\LRloc{1}(\torus\times\Omega)^3$.
A function $\uvel\in\LRloc{1}(\torus\times\Omega)^3$ 
is called \emph{weak solution} to \eqref{sys:NavierStokesTP_ExteriorDomain}
if
\begin{enumerate}
\item[i.]
$\grad\uvel\in\LR{2}(\torus\times\Omega)^{3\times3}$, 
$\uvel\in\LR{2}(\torus;\LR{6}(\Omega))^3$,
$\Div\uvel=0$ in $\torus\times\Omega$,
$\uvel=\uvel_\ast$ on $\torus\times\partial\Omega$,
\label{item:WeakSolution_NStp_L2}
\item[ii.]
$\projcompl\uvel\in\LR{\infty}(\torus;\LR{2}(\Omega))^3$,
\label{item:WeakSolution_NStp_LInfty}
\item[iii.]
the identity
\[
\int_{\torus\times\Omega}\bb{-\uvel\cdot\partial_t\varphi
+\grad\uvel:\grad\varphi
-\rey\partial_1\uvel\cdot\varphi
+\np{\uvel\cdot\grad\uvel}\cdot\varphi}\,\dtx
=\int_{\torus\times\Omega} f\cdot\varphi\,\dtx
\]
holds for all test functions $\varphi\in\CRcisigma(\torus\times\Omega)$.
\label{item:WeakSolution_NStp_WeakFormulation}
\end{enumerate}
\end{defn}

Let us explain the choice of the functional class for weak solutions. 
When $\Omega=\R^3$, condition 
i.~is equivalent to
$\uvel\in\LR{2}(\torus;\DSRNsigma{1}{2}(\R^3))$,
where $\DSRNsigma{1}{2}(\R^3)$ is the closure 
of $\CRcisigma(\R^3)$ with respect to the homogeneous norm $\norm{\grad\cdot}_2$.
In this case, the class of solutions from Definition \ref{def:WeakSolution_NStp}
is the same as considered in \cite{GaldiKyed_TPSolNS3D_AsymptoticProfile}
and \cite{Eiter_SpatiallyAsymptoticStructureTPNS_2020},
where the asymptotic behavior of the velocity field $\uvel$ and its gradient $\grad\uvel$ 
was investigated.
Moreover, for any $f\in\LR{2}(\torus;\DSRN{-1}{2}(\R^3))^3$ 
the existence of a weak solution in the above sense 
was shown by \textsc{Kyed} \cite{Kyed_habil}.
Therefore, this class of solutions is a natural candidate
for further investigation of the associated vorticity field $\curl\uvel$. 

The goal of the present article is to prove the following result.

\begin{thm}\label{thm:VorticityDecay}
Let $\Omega\subset\R^3$ be an exterior domain
with boundary of class $\CR{2}$,
and let $\rey> 0$. 
Let $f$ and $\uvel_\ast$ be such that
\begin{align}
\forall q\in(1,\infty): \ f\in\LR{q}(\torus\times\Omega)^3, 
&& 
\supp f \text{ bounded}
\label{cond:forcing},
\\
\uvel_\ast\in\CR{}\np{\torus;\CR{2}(\partial\Omega)}^3
\cap\CR{1}\np{\torus;(\partial\Omega)}^3,
&&
\int_{\partial\Omega}\uvel_\ast\cdot\nvec\,\dS=0,
\label{cond:BoundaryData}
\end{align}
where $\nvec$ denotes the unit outer normal at $\partial\Omega$.
Let $\uvel$ be a
weak time-periodic solution to \eqref{sys:NavierStokesTP_ExteriorDomain}
in the sense of Definition \ref{def:WeakSolution_NStp}, which satisfies
\begin{equation}
\label{el:VelocityAdditionalIntegrability}
\exists r\in(5,\infty): \quad \projcompl\uvel\in\LR{r}(\torus\times\Omega)^3.
\end{equation}
Then there exist constants $\CVorticityDecay>0$ and 
$\alpha=\alpha(\rey,\per)>0$ such that
\begin{align}
\snorm{\curl\proj\uvel(x)}
&\leq \CVorticityDecay\snorm{x}^{-3/2}\e^{-\alpha\wakefct{x}}, 
\label{est:VorticityDecay_ss}\\
\snorm{\curl\projcompl\uvel(t,x)}
&\leq \CVorticityDecay\snorm{x}^{-9/2}\e^{-\alpha\wakefct{x}}
\label{est:VorticityDecay_pp}
\end{align}
for all $t\in\torus$ and $x\in\Omega$. 
\end{thm}

\begin{rem}The constant $\CVorticityDecay$ depends on $\Omega,\lambda$ and on norms of the solution $u$ which, in turn, can be estimated in terms of the body force $f$. So, ultimately, $\CVorticityDecay$ depends on $\Omega,\lambda$ and $f$. 
If not specified otherwise, this may always be the case for all other constants $C$, $C_i$ that we will introduce throughout the paper.
\end{rem}

\begin{rem} In our proof, we need the zero-flux condition \eqref{cond:BoundaryData}$_4$ on the boundary velocity $u_*$, which, instead, is not needed in the particular case of steady-state solutions \cite{Clark_VorticityAtInfinityNavierStokes_1970,BabenkoVasilev_AsymptoticBehaviorSteadyFlow_1973}. Though it is probable that our result continues to hold if the flux is only ``sufficiently small," it is not clear whether the same conclusion may be drawn for flux of arbitrary magnitude. 

\end{rem}
\begin{rem}Condition \eqref{el:VelocityAdditionalIntegrability}
is merely a technical assumption. 
As pointed out in \cite{GaldiKyed_TPSolNS3D_AsymptoticProfile}
for the case $\Omega=\R^3$,
it leads to additional local regularity of the solution
but does not improve its spatial decay properties.
\end{rem}
\begin{rem} If  $f$ is time-independent, then $u\equiv\calp u$, and our result reduces to that of \textsc{Clark} \cite{Clark_VorticityAtInfinityNavierStokes_1970}
and \textsc{Babenko} and \textsc{Vasil'ev} \cite{BabenkoVasilev_AsymptoticBehaviorSteadyFlow_1973}. Actually --as it becomes clear from our proof--  in such a case, we do not need the assumptions \eqref{cond:BoundaryData}, and \eqref{el:VelocityAdditionalIntegrability}.
\end{rem}

\section{Time-periodic fundamental solutions}
\label{sec:TPFundSol}

In this section, we consider the so-called Oseen linearization of \eqref{sys:NavierStokesTP_ExteriorDomain} in the
whole space
given by
\begin{equation}\label{sys:NSlintp_TPFS}
\begin{pdeq}
\partial_t\uvel-\Delta\uvel - \rey\partial_1\uvel +\grad\upres &= f 
&& \tin\torus\times\R^3, \\
\Div\uvel &=0 && \tin\torus\times\R^3
\end{pdeq}
\end{equation}
for $\rey>0$.
In \cite{GaldiKyed_TPSolNS3D_AsymptoticProfile, EiterKyed_etpfslns},
a velocity fundamental solution $\fundsolvel$ to the time-periodic problem 
\eqref{sys:NSlintp_TPFS} was introduced such that
\[
\uvel=\fundsolvel\ast f
\]
with convolution taken with respect to the 
locally compact abelian group $\grp=\torus\times\R^3$.
It is given by 
\begin{equation}
\fundsolvel \coloneqq \fundsolvelss\otimes \onedist_{\torus} + \fundsolvelpp,
\label{eq:tpfundsol_decompVel}
\end{equation} 
where 
\begin{align}
&\fundsolvelss\colon\R^3\setminus\set{0}\to\R^{3\times3}, \quad 
\fundsolvelssjl(x)
\coloneqq\frac{1}{4\pi\rey}\bb{\delta_{j\ell}\Delta-\partial_j\partial_\ell}
\int_0^{\wakefct{\rey x} /2} \frac{1-\e^{-\tau}}{\tau}\,\dtau, 
\label{eq:OseenFundSolss_3d}\\
&\fundsolvelpp \coloneqq \iFT_\grp\Bb{ 
\frac{1-\delta_\Z(k)}{\snorm{\xi}^2 + i(\perf k - \rey \xi_1)}\,
\Bp{\idmatrix - \frac{\xi\otimes\xi}{\snorm{\xi}^2}}},
\label{eq:tpfundsol_deffundsolcompl}
\end{align}
the symbol $\onedist_{\torus}$ denotes the constant $1$ distribution, 
and 
$\wakefct{x}= \snorm{x}+x_1$ as above.
In particular, the fundamental solution $\fundsolvel$ decomposes into 
a \emph{steady-state} part $\fundsolvelss$ and a \emph{purely periodic} part $\fundsolvelpp$.
The steady-state part $\fundsolvelss$ is the fundamental solution to the steady-state Oseen problem
\begin{align}\label{sys:NSlinss_TPFS}
\begin{pdeq}
-\Delta\vvel - \rey\partial_1\vvel +\grad\vpres &= f && \tin\R^3, \\
\Div\vvel &=0 && \tin\R^3;
\end{pdeq}
\end{align}
see \cite[Section VII.3]{GaldiBookNew}. 
This function shows strongly anisotropic behavior, which is reflected in the pointwise estimates
\begin{equation}
\forall \alpha\in\N_0^3\ \forall \epsilon>0\ \exists C>0\ \forall \snorm{x}\geq \epsilon:\quad  
\snorm{\D_x^\alpha \fundsolvelss(x)} \leq  C\bb{\snorm{x}\np{1+\wakefct{\rey x}}}^{-1-\frac{\snorm{\alpha}}{2}};
\label{est:fundsolss_Decay}
\end{equation}
see \cite[Lemma 3.2]{Farwig_habil}.
For the purely periodic part $\fundsolvelpp$ one can show the estimates
\begin{equation}
\forall \alpha \in \N_0^3 \ 
\forall r\in [1,\infty)\ \forall\epsilon>0\ \exists C>0\ \forall \snorm{x}\geq \epsilon:\  
\norm{\D_x^\alpha \fundsolvelpp(\cdot,x)}_{\LR{r}(\torus)} \leq C\snorm{x}^{-3-\snorm{\alpha}};
\label{est:tpfundsol_ComplPointwiseEst}
\end{equation}
see \cite{EiterKyed_etpfslns}. 
Observe that estimate \eqref{est:tpfundsol_ComplPointwiseEst} does not have an anisotropic character
and that the purely periodic part $\fundsolvelpp$ decays faster than
the steady-state part $\fundsolvelss$.

In order to derive estimates of the solution $\uvel$ from those of the fundamental solution
$\fundsolvel$, one thus has to study convolutions of functions 
that satisfy pointwise estimates similar to those in 
\eqref{est:fundsolss_Decay} and \eqref{est:tpfundsol_ComplPointwiseEst}.
Convolutions of the first type were examined by 
\textsc{Farwig} \cite{Farwig_habil,FarwigOseenAnisotropicallyWeightedSob} 
in dimension $n=3$, 
and later by 
\textsc{Kra\v cmar}, \textsc{Novotn\'y} and \textsc{Pokorn\'y}
\cite{KracmarNovotnyPokorny2001} in the general $n$-dimensional case.
We collect some of their results in the following theorem,
which gives estimates of convolutions with $\fundsolvelss$ and $\grad\fundsolvelss$.

\begin{thm}\label{thm:ConvFundsolss}
Let $A\in[2,\infty)$ and $B\in[0,\infty)$, and let $g\in\LR{\infty}(\R^3)$ 
such that $\snorm{g(x)}\leq M\np{1+\snorm{x}}^{-A}\np{1+\wakefct{x}}^{-B}$. 
Then there exists a constant 
$ C= C(A,B,\rey)>0$
with the following properties:
\begin{enumerate}
\item
If $A+\min\set{1,B}>3$, then
\begin{equation}
\snorml{\snorm{\fundsolvelss}\ast g (x)} 
\leq  C M \bb{\np{1+\snorm{x}}\bp{1+\wakefct{\rey x}}}^{-1}.
\end{equation}
\item
If $A+\min\set{1,B}>3$ and $A+B\geq7/2$, then
\begin{equation}
\snorml{\snorm{\grad\fundsolvelss}\ast g (x)} 
\leq  C M \bb{\np{1+\snorm{x}}\bp{1+\wakefct{\rey x}}}^{-3/2}.
\end{equation}
\end{enumerate}
\end{thm}

\begin{proof}
These are special cases of \cite[Theorems 3.1 and 3.2]{KracmarNovotnyPokorny2001}.
\end{proof}

An analogous result for convolutions with $\fundsolvelpp$ and $\grad\fundsolvelpp$
was derived in \cite{Eiter_SpatiallyAsymptoticStructureTPNS_2020}.

\begin{thm}\label{thm:ConvFundsolpp}
Let $A\in\R$ and $g\in\LR{\infty}(\torus\times\R^3)$ 
such that $\snorm{g(t,x)}\leq M\np{1+\snorm{x}}^{-A}$.
Then for any $\varepsilon>0$ there exists a constant
$C= C(A,\rey,\per,\varepsilon)>0$
with the following properties:
\begin{enumerate}
\item
If $A>3$, then
\begin{equation}
\forall \snorm{x}\geq \epsilon: \qquad
\snorml{\snorm{\fundsolvelpp}\ast_{\grp} g (t,x)} 
\leq C M
\np{1+\snorm{x}}^{-3}.
\label{est:ConvFundsolpp}
\end{equation}
\item
If $A>4$, then
\begin{equation}
\forall \snorm{x}\geq \epsilon: \qquad
\snorml{\snorm{\grad\fundsolvelpp}\ast_{\grp} g (t,x)} 
\leq  C M 
\np{1+\snorm{x}}^{-4}.
\label{est:ConvFundsolpp_Grad}
\end{equation}
\end{enumerate}
\end{thm}

\begin{proof}
We refer to \cite[Theorem 3.3]{Eiter_SpatiallyAsymptoticStructureTPNS_2020}.
\end{proof}

Next we derive a fundamental solution for the vorticity field $\curl\uvel$. 
For $\uvel=\fundsolvel\ast f$ a direct computation yields
\[
\np{\curl\uvel}_m
= \varepsilon_{mhj}\partial_h\fundsolvelssjl\ast\proj f _\ell
+ \varepsilon_{mhj}\partial_h\fundsolvelppjl\ast f_\ell
=\varepsilon_{mh\ell} \partial_h\fundsolvortss\ast\proj f _\ell
+\varepsilon_{mh\ell}\partial_h\fundsolvortpp\ast f_\ell
\]
with
\begin{align}
\fundsolvortss(x)
&\coloneqq \frac{1}{4\pi\snorm{x}}\e^{-\wakefct{\rey x}/2},
\label{eq:fundsolvortss} 
\\
\fundsolvortpp
&\coloneqq \iFT_\grp\Bb{\frac{1-\delta_\Z(k)}{\snorm{\xi}^2-i\rey \xi_1+i\perf k}}.
\label{eq:fundsolvortpp_multiplier}
\end{align}
In conclusion, we obtain
\begin{equation}\label{eq:vorticityOseen}
\curl \uvel(t,x) = \int_{\grp} \grad\fundsolvort(t-s,x-y) \wedge f(s,y) \,\dsy,
\end{equation}
where
\begin{equation}\label{eq:fundsolvort_decomp}
\fundsolvort\coloneqq\fundsolvortss\otimes\onedist_{\torus}+\fundsolvortpp.
\end{equation}
We have thus found an integral formula for the vorticity $\curl\uvel$. 
We call $\fundsolvort$ the \emph{vorticity fundamental solution}.
As for the velocity fundamental solution $\fundsolvel$, 
the vorticity fundamental solution $\fundsolvort$ 
decomposes into a steady-state and a purely periodic part,
which can be analyzed separately.
A direct computation leads to the the following estimate of $\grad\fundsolvortss$.
\begin{thm}\label{thm:FundsolVortss_est}
There exists $ C= C\np{\rey}>0$ 
such that for all $x\in\R^3\setminus\set{0}$ it holds
\begin{equation}
\snorm{\grad\fundsolvortss(x)}
\leq C\bp{\snorm{x}^{-2}+\snorm{x}^{-3/2}\wakefct{\rey x}^{1/2}} 
\e^{-\wakefct{\rey x}/2}.
\label{est:FundsolVortss_grad}
\end{equation}
\end{thm}

\begin{proof}
The estimate follows directly by taking derivatives in \eqref{eq:fundsolvortss}
and using the identity 
$\snorm{\grad\nb{\wakefct{\rey x}}}^2=2\rey^2\wakefct{x}/\snorm{x}$.
\end{proof}

The remainder of this section is dedicated to the derivation of
an analogous estimate of $\grad\fundsolvortpp$.
More precisely, we show the following result.

\begin{thm}\label{thm:FundsolVortpp_est}
There exist constants $C= C(\rey,\per,q,\gamma)>0$
and $\Cexpfs= \Cexpfs(\rey,\per)>0$
such that 
for all $\gamma\in(0,1)$, $q\in[1,\frac{1}{1-\gamma})$ and
$x\in\R^3\setminus\set{0}$
it holds
\begin{align}
\norm{\fundsolvortpp(\cdot,x)}_{\LR{q}(\torus)}
&\leq C\snorm{x}^{-(1+2\gamma)} \e^{- \Cexpfs\snorm{x}},
\label{est:FundsolVortpp_fct}\\
\norm{\grad\fundsolvortpp(\cdot,x)}_{\LR{q}(\torus)}
&\leq C\snorm{x}^{-(2+2\gamma)} \e^{- \Cexpfs\snorm{x}}.
\label{est:FundsolVortpp_grad}
\end{align}
\end{thm}

For the proof of Theorem \ref{thm:FundsolVortpp_est}
we represent $\fundsolvortpp$ in a different way.
From
$\iFT_{\grp}=\iFT_{\torus}\otimes\iFT_{\R^3}$ we conclude the identity
\begin{equation}
\fundsolvortpp(t,x)
= \iFT_\torus \bb{ k\mapsto\bp{1-\delta_\Z(k)}\fundsolaux(x)}(t),
\label{eq:fundsolvortpp_series}
\end{equation}
where 
\[
\fundsolauxeta\coloneqq\iFT_{\R^3}\Bb{\frac{1}{\snorm{\xi}^2-i\rey \xi_1+i\eta}}
\]
is the fundamental solution to the equation
\begin{equation}
\label{eq:HelmholtzWithDrift}
i\eta\,\vvel-\Delta\vvel -\rey \partial_1\vvel =f \qquad \tin\R^3.
\end{equation}
This function is explicitly given by
\begin{equation}\label{eq:fundsolaux3D}
\fundsolauxeta\colon\R^3\setminus\set{0}\to\C, \qquad
\fundsolauxeta(x)=\frac{1}{4\pi\snorm{x}}\e^{i\sqrt{-\mu}\snorm{x}-\frac{\rey}{2}x_1}
\end{equation}
for $\eta\neq0$ and $\mu \coloneqq \mu(\eta,\rey)\coloneqq (\rey/2)^2+i\eta\in\C\setminus\R$;
see \cite[Lemma 3.3]{EiterKyed_etpfslns}.
Here $\sqrt{z}$ is the square root of $z$ with nonnegative imaginary part.
We first derive pointwise estimates of $\fundsolauxeta$.

\begin{lem}\label{lem:fundsolaux_pointwiseEst3D}
Let $\eta_0>0$. Then there exists
$\Cexp= \Cexp(\rey,\eta_0)>0$ such that
\begin{align}
\snorm{\fundsolauxeta(x)}
&\leq C\snorm{x}^{-1}\e^{-\Cexp\snorm{\eta}^\half\snorm{x}}, 
\label{est:fundsolaux_pointwiseEst3D_fct}\\
\snorm{\grad\fundsolauxeta(x)}
&\leq C\bp{\snorm{x}^{-2}+\snorm{\eta}^\half\snorm{x}^{-1}}
\e^{-\Cexp\snorm{\eta}^\half\snorm{x}},
\label{est:fundsolaux_pointwiseEst3D_grad}
\end{align}
for all $\eta\in\R$ with $\snorm{\eta}>\eta_0$ and $x\in\R^3\setminus\set{0}$.
\end{lem}

\begin{proof}
As in \cite[Lemma 3.2]{EiterKyed_etpfslns}, 
we show the existence of a constant
$\Cexp= \Cexp(\rey,\eta_0)>0$ such that 
\[
\impart( \sqrt{-\mu})-\frac{\snorm{\rey}}{2} \geq \Cexp \snorm{\eta}^\frac{1}{2}
\]
for $\snorm{\eta}\geq\eta_0$ and $\mu = (\rey/2)^2+i\eta$.
We thus have
\[
\snormL{\e^{i\sqrt{-\mu}\snorm{x}-\frac{\rey}{2}x_1}}
\leq \e^{-\impart\np{\sqrt{-\mu}}\snorm{x}+\frac{\snorm{\rey}}{2}\snorm{x}}
\leq \e^{-\Cexp\snorm{\eta}^\frac{1}{2}\snorm{x}}.
\]
This directly implies \eqref{est:fundsolaux_pointwiseEst3D_fct}.
Computing derivatives and employing this estimate again, we further deduce
\[
\snorml{\grad\fundsolauxeta(x)}
\leq C\bp{
\snorm{x}^{-2}+\snorm{x}^{-1}\np{\snorm{\sqrt{-\mu}}+\snorm{\rey}}
}\e^{-\Cexp\snorm{\eta}^\frac{1}{2}\snorm{x}},
\]
which implies \eqref{est:fundsolaux_pointwiseEst3D_grad}
by using $\snorm{\rey}\leq2\snorm{\sqrt{-\mu}}\leq C\snorm{\eta}^\frac{1}{2}$
for $\snorm{\eta}\geq\eta_0$.
\end{proof}
Now let
$\cutoff\in\CRi(\R)$, $0\leq \cutoff \leq 1$, with
$\cutoff(\eta)=0$ for $\snorm{\eta}\leq \half$ and $\cutoff(\eta)=1$ for $\snorm{\eta}\geq 1$. 
For $\alpha\in\N_0^3$ with $\snorm{\alpha}\leq1$, $\gamma\in(0,1)$ 
and $x\in\R^3\setminus\set{0}$
define the function
\begin{align}\label{eq:FundsolVortpp_est_multiplier}
\mmultiplier_{\alpha,x}\colon\R\to\R, \qquad 
\mmultiplier_{\alpha,x}(\eta)\coloneqq\cutoff(\eta)\snorm{\eta}^\gamma 
\D^\alpha\fundsolauxetaperf(x).
\end{align}
We show that $\mmultiplier_{\alpha,x}$ is an $\LR{q}(\R)$ multiplier
and give an estimate of the multiplier norm
by means of the Marcinkiewicz Multiplier Theorem.

\begin{lem}\label{lem:FundsolVortpp_est_multiplier}
Let $\alpha\in\N_0^3$ with $\snorm{\alpha}\leq1$, $\gamma\in(0,1)$ 
and $x\in\R^3\setminus\set{0}$.
Then $\mmultiplier_{\alpha,x}$ is an $\LR{q}(\R)$ multiplier for any $q\in(1,\infty)$,
and there exist constants 
$ C
=C\np{\rey,\per,q,\alpha,\gamma}>0$
and $ \CexpfsMult= \CexpfsMult\np{\rey,\per}>0$
such that
\[
\norml{\iFT_\R\bb{\mmultiplier_{\alpha,x}\FT_\R\nb{f}}}_{\LR{q}(\R)}
\leq C\snorm{x}^{-1-\snorm{\alpha}-2\gamma}
\e^{- \CexpfsMult\snorm{x}}\norm{f}_{\LR{q}(\R)}.
\]
\end{lem}

\begin{proof}
At first, let $\alpha=0$.
From \eqref{est:fundsolaux_pointwiseEst3D_fct} we conclude
\[
\snorm{\mmultiplier_{0,x}(\eta)}\leq C \cutoff(\eta)\snorm{\eta}^\gamma
\snorm{x}^{-1}\e^{- \Cexp\snorm{\perf\eta}^\half\snorm{x}}
\leq C \snorm{x}^{-1-2\gamma}\e^{- \Cexp\snorm{\perf\eta}^\half\snorm{x}/2}
\]
for $\snorm{\eta}\geq\half$.
Moreover, differentiating $\fundsolauxetaperf$ with respect to $\eta$, we obtain
\[
\snorml{\partial_\eta\fundsolauxetaperf(x)}
\leq C \snorm{\partial_\eta\sqrt{-\mu}}\,\snorm{x}\,\snorml{\fundsolauxetaperf(x)}
\leq C \snorm{\eta}^{-\half}\snorm{x}\,\snorml{\fundsolauxetaperf(x)},
\]
so that \eqref{est:fundsolaux_pointwiseEst3D_fct} yields
\begin{align*}
\snorm{\eta\partial_\eta\mmultiplier_{0,x}(\eta)}
&\leq\snorml{\cutoff'(\eta)\snorm{\eta}^{\gamma+1}\fundsolauxetaperf(x)}
+\snorml{\cutoff(\eta)\gamma\snorm{\eta}^\gamma\fundsolauxetaperf(x)}
+\snorml{\cutoff(\eta)\snorm{\eta}^{\gamma+1}\partial_\eta\fundsolauxetaperf(x)}\\
&\quad\leq C \bp{\snorm{\eta}^\gamma\snorm{x}^{-1}+\snorm{\eta}^{\gamma+\half}} 
\e^{- \Cexp\snorm{\perf\eta}^\half\snorm{x}}
\leq C \snorm{x}^{-1-2\gamma}\e^{- \Cexp\snorm{\perf\eta}^\half\snorm{x}/2}
\end{align*}
for $\snorm{\eta}\geq\half$.
Collecting these estimates and utilizing $\mmultiplier_{0,x}(\eta)=0$ for 
$\snorm{\eta}\leq \half$, we have
\begin{equation}\label{est:fundsolvortpp_fct_multiplierEst}
\snorm{\mmultiplier_{0,x}(\eta)}+\snorm{\eta\partial_\eta\mmultiplier_{0,x}(\eta)}
\leq C \snorm{x}^{-1-2\gamma}\e^{- \CexpfsMult\snorm{x}}
\end{equation}
with $ \CexpfsMult=\sqrt{\pi/\per} \Cexp/2$ for all $\eta\in\R$.

Next consider the case $\alpha=\ej$ for some $j\in\set{1,2,3}$.
Then \eqref{est:fundsolaux_pointwiseEst3D_grad} leads to
\[
\snorm{\mmultiplier_{\alpha,x}(\eta)}
\leq C \cutoff(\eta)\snorm{\eta}^\gamma
\bp{\snorm{x}^{-2}+\snorm{\eta}^\half\snorm{x}^{-1}}
\e^{- \Cexp\snorm{\perf\eta}^\half\snorm{x}}
\leq C \snorm{x}^{-2-2\gamma}\e^{- \Cexp\snorm{\perf\eta}^\half\snorm{x}/2}
\]
for $\snorm{\eta}\geq\half$.
Moreover, a straightforward calculation yields
\[
\snorml{\partial_\eta\partial_j\fundsolauxetaperf(x)}
\leq C \bp{\snorm{\mu}^{-\half}+\snorm{x}}
\snorml{\fundsolauxetaperf(x)}
\leq C \bp{\snorm{\eta}^{-\half}+\snorm{x}}
\snorml{\fundsolauxetaperf(x)},
\]
so that we can employ Lemma \ref{lem:fundsolaux_pointwiseEst3D} to estimate
\begin{align*}
\snorm{\eta\partial_\eta\partial_j\mmultiplier_{\alpha,x}(\eta)}
&\leq\snorml{\cutoff'(\eta)\snorm{\eta}^{\gamma+1}\partial_j\fundsolauxetaperf(x)}
+\snorml{\cutoff(\eta)\gamma\snorm{\eta}^\gamma\partial_j\fundsolauxetaperf(x)}\\
&\qquad\qquad+\snorml{\cutoff(\eta)\snorm{\eta}^{\gamma+1}
\partial_\eta\partial_j\fundsolauxetaperf(x)}\\
&\leq C \bp{\snorm{\eta}^\gamma\snorm{x}^{-2}+\snorm{\eta}^{\gamma+\half}\snorm{x}^{-1}
+\snorm{\eta}^{\gamma+1}} 
\e^{- \Cexp\snorm{\perf\eta}^\half\snorm{x}}
\\
&\leq C \snorm{x}^{-2-2\gamma}\e^{- \Cexp\snorm{\perf\eta}^\half\snorm{x}/2}
\end{align*}
for $\snorm{\eta}\geq\half$.
Collecting these estimates and utilizing $\mmultiplier_{\alpha,x}(\eta)=0$ for 
$\snorm{\eta}\leq \half$, we have
\begin{align}\label{est:fundsolvortpp_grad_multiplierEst}
\snorm{\mmultiplier_{\alpha,x}(\eta)}+\snorm{\eta\partial_\eta\mmultiplier_{\alpha,x}(\eta)}
\leq C \snorm{x}^{-2-2\gamma}\e^{- \CexpfsMult\snorm{x}}
\end{align}
with $ \CexpfsMult=\sqrt{\pi/\per} \Cexp/2$ as above.

By the Marcinkiewicz Multiplier Theorem (see \cite[Corollary 5.2.5]{Grafakos1}), 
the assertion is now a direct consequence of 
\eqref{est:fundsolvortpp_fct_multiplierEst} and \eqref{est:fundsolvortpp_grad_multiplierEst}.
\end{proof}

Using this result, we establish the pointwise estimates of $\fundsolvortpp$ asserted in Theorem \ref{thm:FundsolVortpp_est}
by means of the so-called transference principle for Fourier multipliers.

\begin{proof}[Proof of Theorem \ref{thm:FundsolVortpp_est}]
It suffices to consider $q\in(1,\infty)$.
Due to \eqref{eq:fundsolvortpp_series}, we have
\begin{equation}\label{eq:FundsolVortpp_TorusMult}
\D_x^\alpha\fundsolvortpp(\cdot,x)= \iFT_\torus\bb{\Mmultiplier_{\alpha,x}\FT_\torus\nb{\varphi_\gamma}}
\end{equation}
with
\[
\Mmultiplier_{\alpha,x}(k)\coloneqq \bp{1-\delta_\Z(k)} \snorm{k}^\gamma \D_x^\alpha\fundsolaux (x),
\qquad
\varphi_\gamma\coloneqq\iFT_\torus\bb{k\mapsto\bp{1-\delta_\Z(k)}\snorm{k}^{-\gamma}}.
\]
First, note that $\Mmultiplier_{\alpha,x}=\restriction{\mmultiplier_{\alpha,x}}{\Z}$.
Since $\mmultiplier_{\alpha,x}$ is a continuous $\LR{q}(\R)$ multiplier
by Lemma \ref{lem:FundsolVortpp_est_multiplier},
the transference principle 
(see \cite[Theorem B.2.1]{EdwardsGaudryBook} or \cite[Theorem 2.15]{EiterKyed_tplinNS_PiFbook})
implies that $\Mmultiplier_{\alpha,x}$
is an $\LR{q}(\torus)$ multiplier for any $q\in(1,\infty)$
and that
\[
\norm{\iFT_\torus\bb{\Mmultiplier_{\alpha,x}\FT_\torus\nb{f}}}_{\LR{q}(\torus)}
\leq C \snorm{x}^{-1-\snorm{\alpha}-2\gamma}\e^{- \Cexpfs\snorm{x}}\norm{f}_{\LR{q}(\torus)}.
\]
Moreover, we have
$\varphi_\gamma\in\LR{q}(\torus)$
provided $q<1/(1-\gamma)$, which is a direct consequence of
\cite[Example 3.1.19]{Grafakos1} for example.
Finally, the assertion follows from \eqref{eq:FundsolVortpp_TorusMult}.
\end{proof}

\section{Regularity results}\label{sec:Regularity}
 
Here we collect some results concerning the
regularity of weak solutions to \eqref{sys:NavierStokesTP_ExteriorDomain}
and its linearization,
which is given by
\begin{equation}\label{sys:NSlintp_extdomain}
\begin{pdeq}
\partial_t\uvel-\Delta\uvel - \rey\partial_1\uvel +\grad\upres &= f 
&& \tin\torus\times\Omega, 
\\
\Div\uvel &=0 && \tin\torus\times\Omega,
\\
\uvel & =\uvel_\ast
&&\ton\torus\times\partial\Omega.
\end{pdeq}
\end{equation}
First of all, we derive the following regularity theorem for solutions to 
\eqref{sys:NSlintp_extdomain} 
in the case $\proj f=0$.

\begin{lem}\label{lem:NSlintp_Regularity}
Let $\Omega\subset\R^3$ be an exterior domain of class $\CR{2}$,
let $\uvel_\ast$ be as in \eqref{cond:BoundaryData},
and let $f\in\LR{q}(\torus\times\Omega)$ for some $q\in(1,\infty)$.
Assume $\proj f=0$ and that $\uvel$ is a weak 
solution to \eqref{sys:NSlintp_extdomain},
that is, $\uvel=\uvel_\ast$ on $\torus\times\partial\Omega$,
$\Div\uvel=0$ and 
\begin{equation}\label{eq:NSlintp_weak}
\int_{\torus\times\Omega}\bb{-\uvel\cdot\partial_t\varphi
+\grad\uvel:\grad\varphi
-\rey\partial_1\uvel\cdot\varphi
}\,\dtx
=\int_{\torus\times\Omega} f\cdot\varphi\,\dtx
\end{equation}
for all $\varphi\in\CRcisigma(\torus\times\Omega)$.
Assume that $\uvel\in\LR{\infty}(\torus;\LR{2}(\Omega))^3$ and
$\grad\uvel\in\LR{2}(\torus\times\Omega)^{3\times3}$.
Then 
$\uvel\in\WSR{1,2}{q}(\torus\times\Omega)^3$,
and there exists $\upres\in\LR{q}(\torus;\DSR{1}{q}(\Omega))$
such that $\np{\uvel,\upres}$ is a strong solution to
\eqref{sys:NSlintp_extdomain}.
\end{lem}

\begin{proof}
First of all, using classical arguments 
(see \cite[Section III.3]{GaldiBookNew} for example)
one can show the existence of a function 
$\Uvel\in\WSR{1,2}{q}(\torus\times\Omega)^3$
such that $\Uvel=\uvel_\ast$ on $\torus\times\partial\Omega$ and
$\Div\Uvel=0$.
Moreover, since $\proj f=0$, 
by \cite[Theorem 5.1]{GaldiKyed_TPflowViscLiquidpBody}
there exist $\zvel\in\WSR{1,2}{q}(\torus\times\Omega)^3$ and
$\upres\in\LR{q}(\torus;\DSR{1}{q}(\Omega))$
such that
\begin{equation}
\begin{pdeq}
\partial_t\zvel-\Delta\zvel - \rey\partial_1\zvel +\grad\upres 
&= f - \partial_t\Uvel-\Delta\Uvel - \rey\partial_1\Uvel
&& \tin\torus\times\Omega, 
\\
\Div\zvel &=0 && \tin\torus\times\Omega,
\\
\zvel & =0
&&\ton\torus\times\partial\Omega.
\end{pdeq}
\end{equation}
Then $\np{\tuvel,\upres}\coloneqq\np{\zvel+\Uvel,\upres}$
solves \eqref{sys:NSlintp_extdomain},
and for the completion of the proof it remains to show $\uvel=\tuvel$.
For this purpose, we employ a duality argument.
Let $\psi\in\CRci(\torus\times\Omega)^3$.
By \cite[Theorem 5.1]{GaldiKyed_TPflowViscLiquidpBody}
there exist functions
$\wvel\in\WSR{1,2}{2}(\torus\times\Omega)^3
\cap\WSR{1,2}{q'}(\torus\times\Omega)^3$
and 
$\wpres\in\LR{2}(\torus;\DSR{1}{2}(\Omega))
\cap\LR{q'}(\torus;\DSR{1}{q'}(\Omega))$,
where $q'=q/(q-1)$,
which satisfy
\begin{equation}\label{sys:NSlintp_extdomain_hombdry}
\begin{pdeq}
\partial_t\wvel-\Delta\wvel + \rey\partial_1\wvel +\grad\wpres &= \projcompl\psi 
&& \tin\torus\times\Omega, 
\\
\Div\wvel &=0 && \tin\torus\times\Omega,
\\
\wvel & =0
&&\ton\torus\times\partial\Omega.
\end{pdeq}
\end{equation}
By a standard density argument one shows that we can let $\varphi=\wvel$ in 
the weak formulation \eqref{eq:NSlintp_weak}.
Then, by an integration by parts, we get
\[
\begin{aligned}
\int_{\torus\times\Omega}\np{\uvel-\tuvel}\cdot\projcompl\psi\,\dtx
&=\int_{\torus\times\Omega}\np{\uvel-\tuvel}
\cdot\bp{\partial_t\wvel-\Delta\wvel + \rey\partial_1\wvel +\grad\wpres}\,\dtx
\\
&=\int_{\torus\times\Omega}\bb{\uvel\cdot\partial_t\wvel+\grad\uvel:\grad\wvel
- \rey\partial_1\uvel\cdot\wvel}\,\dtx
\\
&\qquad\qquad-\int_{\torus\times\Omega} \bp{\partial_t\tuvel-\Delta\tuvel - \rey\partial_1\tuvel +\grad\upres}\cdot\wvel\,\dtx 
\\
&=\int_{\torus\times\Omega} f\cdot\wvel-\int_{\torus\times\Omega} f\cdot\wvel
=0.
\end{aligned}
\]
Since $\proj\uvel=\proj\tuvel=0$, we thus conclude
\[
\int_{\torus\times\Omega}\np{\uvel-\tuvel}\cdot\psi\,\dtx
=\int_{\torus\times\Omega}\np{\uvel-\tuvel}\cdot\proj\psi\,\dtx
+\int_{\torus\times\Omega}\np{\uvel-\tuvel}\cdot\projcompl\psi\,\dtx
=0
\]
for arbitrary $\psi\in\CRci\np{\torus\times\Omega}^3$,
which implies $\uvel=\tuvel$ and completes the proof.
\end{proof}

Based on this result for the linearized problem \eqref{sys:NSlintp_extdomain}, 
we can now show that the additional integrability condition
assumed in \eqref{el:VelocityAdditionalIntegrability}
leads to higher regularity of the weak solution.

\begin{lem}\label{lem:NavierStokesTP_Regularity}
In the hypotheses of Theorem \ref{thm:VorticityDecay}
we have $\uvel=\vvel+\wvel$, with 
$\vvel=\proj\uvel$, $\wvel=\projcompl\uvel$ such that
\begin{align}
&\forall q\in(1,\infty):\, \vvel\in\DSR{2}{q}(\Omega),\ \
\forall r\in(\frac{4}{3},\infty]:\, \vvel\in\DSR{1}{r}(\Omega), \ \
\forall s\in(2,\infty]:\, \vvel\in\LR{s}(\Omega),
\label{el:NavierStokesTP_Regularity_ss}
\\
&\forall q\in(1,\infty): \ \wvel\in\WSR{1,2}{q}(\torus\times\Omega).
\label{el:NavierStokesTP_Regularity_pp}
\end{align}
Moreover, there exists a pressure field $\upres$ with 
\begin{equation}\label{el:NavierStokesTP_Regularity_pressure}
\forall q\in(1,\infty): \ \upres\in\LR{q}(\torus;\DSR{1}{q}(\Omega))
\end{equation}
such that \eqref{sys:NavierStokesTP_ExteriorDomain} 
is satisfied in the strong sense.
\end{lem}

\begin{proof}
At first, observe that $\vvel$ is a weak solution to
the steady-state Navier--Stokes problem
\begin{equation}\label{sys:NavierStokesTP_ExteriorDomain_sspart}
\begin{pdeq}
-\Delta\vvel-\rey\partial_1\vvel+\vvel\cdot\grad\vvel+\grad\vpres
&=\proj f-\proj\nb{\wvel\cdot\grad\wvel}
&&\tin\Omega, \\
\Div\vvel&=0 
&&\tin\Omega, \\
\vvel&=\proj\uvel_\ast 
&&\ton\partial\Omega.
\end{pdeq}
\end{equation}
H\"older's inequality and Definition \ref{def:WeakSolution_NStp} yield
$\wvel\cdot\grad\wvel
\in\LR{1}(\torus;\LR{3/2}(\Omega))\cap\LR{2}(\torus;\LR{1}(\Omega))$.
Therefore, we have
$\proj f-\proj(\wvel\cdot\grad\wvel)\in\LR{1}(\Omega)\cap\LR{3/2}(\Omega)$,
and \cite[Lemma X.6.1]{GaldiBookNew} implies
\begin{equation}\label{el:regularity_ss1}
\forall q\in(1,\frac{3}{2}]:\, \vvel\in\DSR{2}{q}(\Omega),\ \
\forall r\in(\frac{4}{3},3]:\, \vvel\in\DSR{1}{r}(\Omega), \ \
\forall s\in(2,\infty):\, \vvel\in\LR{s}(\Omega)
\end{equation}
and the existence of $\vpres\in\DSR{1}{q}(\Omega)$ for all $q\in(1,\frac{3}{2}]$
such that \eqref{sys:NavierStokesTP_ExteriorDomain_sspart} is satisfied 
in the strong sense.
Moreover, 
$\wvel$ is a weak solution to
\begin{equation}\label{sys:NavierStokesTP_ExteriorDomain_pppart}
\begin{pdeq}
\partial_t\wvel-\Delta\wvel - \rey\partial_1\wvel +\grad\wpres 
&= \projcompl f -\vvel\cdot\grad\wvel-\wvel\cdot\grad\vvel
-\projcompl\np{\wvel\cdot\grad\wvel}
&& \tin\torus\times\Omega, 
\\
\Div\wvel &=0 && \tin\torus\times\Omega,
\\
\wvel & =\projcompl\uvel_\ast
&&\ton\torus\times\partial\Omega.
\end{pdeq}
\end{equation}
By a standard interpolation argument, the assumptions from 
Definition \ref{def:WeakSolution_NStp} imply
$\wvel\in\LR{10/3}(\torus\times\Omega)$.
Since $\vvel\in\LR{10/3}(\Omega)$ by \eqref{el:regularity_ss1}
and $\grad\uvel\in\LR{2}(\torus\times\Omega)$ by assumption,
this implies
\begin{equation}\label{el:rhs_pp}
\projcompl f -\vvel\cdot\grad\wvel-\wvel\cdot\grad\vvel
-\projcompl\np{\wvel\cdot\grad\wvel} \in\LR{s}(\torus\times\Omega)
\end{equation}
for $s=5/4$.
Now Lemma \ref{lem:NSlintp_Regularity} shows 
that there exists $\wpres$
such that
\begin{equation}\label{el:regularity_pp_5/4}
\wvel\in\WSR{1,2}{5/4}(\torus\times\Omega), \ 
\wpres\in\LR{5/4}(\torus;\DSR{1}{5/4}(\Omega)),
\end{equation}
and \eqref{sys:NavierStokesTP_ExteriorDomain_pppart}
holds in a strong sense.
Starting from \eqref{el:regularity_pp_5/4},
we now employ a boot-strap argument to conclude the proof.

If $\wvel\in\WSR{1,2}{q}(\torus\times\Omega)$ 
for some $q\in(1,15/8)$,
then the embedding theorem from 
\cite[Theorem 4.1]{GaldiKyed_TPflowViscLiquidpBody}
implies $\grad\wvel\in\LR{5q/(5-q)}(\torus\times\Omega)$. 
In virtue of \eqref{el:VelocityAdditionalIntegrability} 
and \eqref{el:regularity_ss1},
this implies $\wvel\cdot\grad\wvel,\
\vvel\cdot\grad\wvel\in\LR{s}(\torus\times\Omega)$
for $\frac{1}{s}=\frac{1}{q}+\frac{1}{r}-\frac{1}{5}$.
Moreover, \cite[Theorem 4.1]{GaldiKyed_TPflowViscLiquidpBody}
yields $\wvel\in\LR{5q/(5-q)}(\torus;\LR{15q/(15-8q)}(\Omega))$,
so that
$\wvel\cdot\grad\vvel\in\LR{s}$ by \eqref{el:regularity_ss1}.
In total, we thus obtain 
\eqref{el:rhs_pp} for $\frac{1}{s}=\frac{1}{q}+\frac{1}{r}-\frac{1}{5}$,
and Lemma \ref{lem:NSlintp_Regularity} leads to the implication
\begin{equation}\label{impl:bootstrap_1}
\exists\, q\in\bp{1,\frac{15}{8}}:\, \wvel\in\WSR{1,2}{q}(\torus\times\Omega)
\implies
\forall\, \frac{1}{s}\in\bb{\frac{1}{q}+\frac{1}{r}-\frac{1}{5},\frac{1}{q}}:\,
\wvel\in\WSR{1,2}{s}(\torus\times\Omega).
\end{equation}

If $\wvel\in\WSR{1,2}{q}(\torus\times\Omega)$ 
for some $q\in[5/3,5/2)$,
then \cite[Theorem 4.1]{GaldiKyed_TPflowViscLiquidpBody}
yields $\wvel\in\LR{5q/(5-2q)}(\torus\times\Omega)$ and
$\grad\wvel\in\LR{5q/(5-q)}(\torus\times\Omega)$,
which implies $\wvel\cdot\grad\wvel\in\LR{s_1}(\torus\times\Omega)$
for all $\frac{1}{s_1}\in[\frac{2}{q}-\frac{3}{5},\frac{2}{q}]$.
Hence we have $\proj(\wvel\cdot\grad\wvel)\in\LR{s_1}(\Omega)$,
and another application of \cite[Lemma X.6.1]{GaldiBookNew} in view of 
\eqref{el:regularity_ss1} yields $\grad\vvel\in\LR{t}(\Omega)$ for all $t\in[4/3,15/4]$ and $\vvel\in\LR{t}(\Omega)$ for all $t\in(2,\infty]$.
We thus conclude
$\wvel\cdot\grad\vvel\in\LR{s_2}(\torus\times\Omega)$ for 
$\frac{1}{s_2}\in[\frac{1}{q}-\frac{2}{15},\frac{1}{q}+\frac{3}{4}]$
and $\vvel\cdot\grad\wvel\in\LR{s_3}(\torus\times\Omega)$ 
for $\frac{1}{s_3}\in[\frac{1}{q}-\frac{1}{5},\frac{1}{q}+\frac{1}{2}]$.
In particular, we obtain \eqref{el:rhs_pp}
for $\frac{1}{s}=\frac{2}{q}-\frac{3}{5}$ if $q\leq15/7$,
and $\frac{1}{s}=\frac{1}{q}-\frac{2}{15}$ if $q\geq15/7$.
By Lemma \ref{lem:NSlintp_Regularity}, this implies
\begin{align}
\exists\, q\in\big(\frac{5}{3},\frac{15}{7}\big]:\, \wvel\in\WSR{1,2}{q}(\torus\times\Omega)
\implies
&\forall\, \frac{1}{s}\in\bb{\frac{2}{q}-\frac{3}{5},\frac{1}{q}}:\,
\wvel\in\WSR{1,2}{s}(\torus\times\Omega),
\label{impl:bootstrap_2} 
\\
\exists\, q\in\big[\frac{15}{7},\frac{5}{2}\big):\, \wvel\in\WSR{1,2}{q}(\torus\times\Omega)
\implies
&\forall\, \frac{1}{s}\in\bb{\frac{1}{q}-\frac{2}{15},\frac{1}{q}}:\,
\wvel\in\WSR{1,2}{s}(\torus\times\Omega).
\label{impl:bootstrap_3}
\end{align}

Arguing in a similar fashion, one shows the further implications
\begin{align}
\exists\, q\in\big[\frac{5}{2},5):\, \wvel\in\WSR{1,2}{q}(\torus\times\Omega)
\implies
&\forall\, \frac{1}{s}\in\big(\frac{1}{q}-\frac{1}{5},\frac{1}{q}\big]:\,
\wvel\in\WSR{1,2}{s}(\torus\times\Omega),
\label{impl:bootstrap_4} 
\\
\exists\, q\in[5,\infty):\, \wvel\in\WSR{1,2}{q}(\torus\times\Omega)
\implies
&\forall\, s\in[q,\infty):\,
\wvel\in\WSR{1,2}{s}(\torus\times\Omega).
\label{impl:bootstrap_5}
\end{align}
Using now \eqref{el:regularity_pp_5/4} as starting point, 
we can iteratively employ \eqref{impl:bootstrap_1}--\eqref{impl:bootstrap_5}
to obtain $\wvel\in\WSR{1,2}{s}(\torus\times\Omega)$ 
for all $s\in[5/4,\infty)$.
Firstly, this 
yields $\wvel\cdot\grad\wvel\in\LR{\infty}(\torus\times\Omega)$,
so that
$\proj f-\proj\np{\wvel\cdot\grad\wvel}\in\LR{q}(\Omega)$ 
for all $q\in[1,\infty)$.
Now \eqref{el:NavierStokesTP_Regularity_ss} 
is a direct consequence of \cite[Theorem X.6.4]{GaldiBookNew}.
Secondly, this shows that
\eqref{el:rhs_pp} holds for all $s\in[1,\infty)$,
whence Lemma \ref{lem:NSlintp_Regularity} 
implies \eqref{el:NavierStokesTP_Regularity_pp}.
Finally, the claimed regularity \eqref{el:NavierStokesTP_Regularity_pressure}
of $\upres=\vpres+\wpres$ is
a direct consequence.
This completes the proof.
\end{proof}
 
\section{The fixed-point problem}
\label{sec:FixedPointProblem}

In this section we derive a suitable fixed-point equation 
satisfied by weak solutions in the whole space,
and we introduce the necessary functional framework.
More precisely, the main focus of the subsequent analysis
lies on the study of problem \eqref{sys:NavierStokesTP_ExteriorDomain}
when $\Omega=\R^3$, namely, 
\begin{equation}\label{sys:NavierStokesTP}
\begin{pdeq}
\partial_t\uvel-\Delta\uvel-\rey\partial_1\uvel+\uvel\cdot\grad\uvel+\grad\upres&=f
&&\tin\torus\times\R^3, \\
\Div\uvel&=0 
&&\tin\torus\times\R^3, \\
\lim_{\snorm{x}\to\infty}\uvel(t,x)&=0 
&&\tfor t\in\torus\,.
\end{pdeq}
\end{equation}
The case of an exterior domain will be treated
at the end of the last section.
\par
We begin to observe that asymptotic properties of weak solutions to \eqref{sys:NavierStokesTP}
were studied in 
\cite{GaldiKyed_TPSolNS3D_AsymptoticProfile}
and \cite{Eiter_SpatiallyAsymptoticStructureTPNS_2020},
where the following
decay estimates of $\uvel$ and $\grad\uvel$ were derived.

\begin{thm}\label{thm:VelocityDecay}
Let $\rey>0$ and $f\in\LR{q}(\torus\times\R^3)^3$ for all $q\in(1,\infty)$
and let $\supp f$ be compact. 
Let $\uvel$ be a weak solution to \eqref{sys:NavierStokesTP} that satisfies \eqref{el:VelocityAdditionalIntegrability} 
for $\Omega=\R^3$.
Then there is $ \CVelocityDecay>0$
such that for all $\np{t,x}\in\torus\times\R^3$ 
the function $\uvel$ satisfies
\begin{align}
\snorm{\proj\uvel(x)}
&\leq \CVelocityDecay\bb{\bp{1+\snorm{x}}\bp{1+\wakefct{\rey x}}}^{-1}, 
\label{est:VelocityDecay_fct_ss}
\\
\snorm{\grad\proj\uvel(x)}
&\leq \CVelocityDecay\bb{\bp{1+\snorm{x}}\bp{1+\wakefct{\rey x}}}^{-\frac{3}{2}},
\label{est:VelocityDecay_grad_ss}
\\
\snorm{\projcompl\uvel(t,x)}
&\leq \CVelocityDecay\bp{1+\snorm{x}}^{-3},
\label{est:VelocityDecay_fct_pp}
\\
\snorm{\grad\projcompl\uvel(t,x)}
&\leq \CVelocityDecay\bp{1+\snorm{x}}^{-4}.
\label{est:VelocityDecay_grad_pp}
\end{align}
\end{thm}

\begin{proof}
Under the assumption $f\in\CRci(\torus\times\R^3)^3$,
this result was shown in \cite[Theorem 4.5]{Eiter_SpatiallyAsymptoticStructureTPNS_2020}
based on estimates of the velocity field $\uvel$ derived in \cite{GaldiKyed_TPSolNS3D_AsymptoticProfile}.
However, a careful study of the proofs shows 
that these results continue to be valid under the stated weaker assumption on $f$.
\end{proof}

To derive a suitable fixed-point equation, 
we exploit the following representation formulas
that result from the time-periodic fundamental solutions introduced in the previous section.

\begin{prop}\label{prop:RepresentationVorticity}
Let $\uvel$ be a weak solution as in Theorem \ref{thm:VelocityDecay}. 
Then 
\begin{equation}\label{eq:RepresentationVelocity_CurluWedgeu}
D_x^\alpha\uvel=D_x^\alpha\fundsolvel\ast\nb{f-\curl\uvel\wedge\uvel}
\end{equation}
for all $\alpha\in\N_0^3$ with $\snorm{\alpha}\leq1$.
In particular, the steady-state part
$\vvel\coloneqq\proj\uvel$ and the purely periodic part $\wvel\coloneqq\projcompl\uvel$ satisfy 
\begin{align}
D_x^\alpha\vvel&=D_x^\alpha\fundsolvelss\ast\bb{\proj f - \curl\vvel\wedge\vvel
-\proj\np{\curl\wvel\wedge\wvel}},
\label{eq:RepresentationVelocity_CurluWedgeu_ss}
\\
D_x^\alpha\wvel&=D_x^\alpha\fundsolvelpp\ast
\bb{\projcompl f -\curl\vvel\wedge\wvel -\curl\wvel\wedge\vvel+\projcompl\np{\curl\wvel\wedge\wvel}}.
\label{eq:RepresentationVelocity_CurluWedgeu_pp}
\end{align}
Moreover, we have
\begin{equation}\label{eq:RepresentationVorticity_CurluWedgeu}
\curl \uvel(t,x)
=\int_{\grp} \grad\fundsolvort(t-s,x-y) \wedge \bb{f-\curl\uvel\wedge\uvel}(s,y)\,\dsy,
\end{equation}
and
\begin{equation}
\curl\vvel(x)
=\int_{\R^3} \grad\fundsolvortss(x-y) \wedge \bb{\proj f - \curl\vvel\wedge\vvel
-\proj\np{\curl\wvel\wedge\wvel}}(y)\,\dy,
\label{eq:RepresentationVorticity_CurluWedgeu_ss}
\end{equation}
as well as
\begin{equation}
\begin{aligned}
\curl\wvel(t,x)
&=\int_{\grp} \grad\fundsolvortpp(t-s,x-y) \wedge \bb{\projcompl f -\curl\vvel\wedge\wvel \\
&\qquad\qquad\qquad-\curl\wvel\wedge\vvel-\projcompl\np{\curl\wvel\wedge\wvel}}(s,y)\,\dsy.
\end{aligned}
\label{eq:RepresentationVorticity_CurluWedgeu_pp}
\end{equation}
\end{prop}
\begin{proof}
Since $\uvel\cdot\grad\uvel=\frac{1}{2}\grad\bp{\snorm{\uvel}^2}+\curl\uvel\wedge\uvel$ and
$\fundsolvel\ast\grad\bp{\snorm{\uvel}^2}=\Div\bp{\fundsolvel\ast\snorm{\uvel}^2}=0$,
the equations \eqref{eq:RepresentationVelocity_CurluWedgeu},
\eqref{eq:RepresentationVelocity_CurluWedgeu_ss} and \eqref{eq:RepresentationVelocity_CurluWedgeu_pp} 
are direct consequences of \cite[Proposition 4.8]{Eiter_SpatiallyAsymptoticStructureTPNS_2020}.
The remaining identities follow by 
applying the $\curl$ operator to both sides of these formulas and repeating the computations from Section \ref{sec:TPFundSol}.
\end{proof}

\begin{rem}\label{rem:VortDecay_pp_NotAsInLinear}
In view of Proposition \ref{prop:RepresentationVorticity}
and the pointwise estimates of $\fundsolvort$ from
Theorem \ref{thm:FundsolVortss_est} and Theorem \ref{thm:FundsolVortpp_est},
 we can explain, at this point, the origin of the pointwise estimates 
stated in Theorem \ref{thm:VorticityDecay}.
Comparing \eqref{est:VorticityDecay_ss} and \eqref{est:FundsolVortss_grad},
we see that the asserted decay rates of the steady-state parts $\curl\vvel$ 
and $\grad\fundsolvortss$ coincide, 
which is the optimal result one can expect to derive 
from equation \eqref{eq:RepresentationVorticity_CurluWedgeu_ss}
for general $f\in\CRci(\grp)^3$.
In contrast, the asserted decay rates of the purely periodic parts 
$\curl\wvel$ and $\fundsolvelpp$ given in 
\eqref{est:VorticityDecay_pp} and \eqref{est:FundsolVortpp_grad}, respectively,
do not coincide.
The reason is due to the presence of the term $\curl\vvel\wedge\wvel$ 
in \eqref{eq:RepresentationVorticity_CurluWedgeu_pp}. 
By assuming the---to some extent---optimal decay rate \eqref{est:VorticityDecay_ss} for $\curl\vvel$,
the pointwise estimate of $\wvel$ from \eqref{est:VelocityDecay_fct_pp} implies
\[
\snorm{\curl\vvel\wedge\wvel}(t,x)\leq C \snorm{x}^{-9/2}\e^{- \alpha \wakefct{\rey x}}.
\]
In the end, this term dominates the decay of the right-hand side 
of \eqref{eq:RepresentationVorticity_CurluWedgeu_pp} 
and thus the pointwise estimates of $\curl\wvel$. As a result,
 the decay of $\curl\wvel$ is slower than that of $\grad\fundsolvortpp$ but, however, still faster than the decay rate of the steady-state vorticity field $\curl\vvel$.
\end{rem}

Proposition \ref{prop:RepresentationVorticity} yields fixed-point equations for $\uvel$
and $\curl\uvel$ and the respective steady-state and purely periodic parts,
which we now decompose in an appropriate way.
Let $\cutoff\in\CRci(\R;[0,1])$ 
with $\cutoff(s)=1$ for $\snorm{s}\leq 5/4$ 
and $\cutoff(s)=0$ for $\snorm{s}\geq 7/4$.
For $\radius>0$ define
$\cutoff_\radius\in\CRci(\R^3;[0,1])$ by $\cutoff_{\radius}(x)\coloneqq\cutoff(\radius^{-1}\snorm{x})$,
and fix $\radius_0>0$ such that $\supp f\subset \torus\times\ball_{S_0}$.
For $\radius\in[2\radius_0,\infty)$
we express \eqref{eq:RepresentationVelocity_CurluWedgeu} 
as the sum of two terms, namely
\[
\uvel
=\fundsolvel\ast\bb{-\np{1-\cutoff_\radius}\curl\uvel\wedge\uvel}
+\fundsolvel\ast\bb{f-\cutoff_\radius\curl\uvel\wedge\uvel}.
\]
Due to $\supp\np{1-\cutoff_\radius}\subset\ball^\radius$, 
this yields 
\begin{align}\label{eq:uFixedPoint_VortDecay}
\restriction{\uvel}{\torus\times\ball^\radius}=\calf_\radius(\restriction{\uvel}{\torus\times\ball^\radius})+\calh_\radius,
\end{align}
where
\[
\begin{aligned}
\calf_\radius(\zvel)
&\coloneqq\restriction{\bp{\fundsolvel\ast\bb{-\np{1-\cutoff_\radius}\curl\zvel\wedge\zvel}}}{\torus\times\ball^\radius}, \\
\calh_\radius
&\coloneqq\restriction{\bp{\fundsolvel\ast\bb{f-\cutoff_\radius\curl\uvel\wedge\uvel}}}{\torus\times\ball^\radius}.
\end{aligned}
\]
We set $\nonlin(\zvel)\coloneqq-\curl\zvel\wedge\zvel$ and
\begin{align}
\nonlinss(\zvel)
\coloneqq\proj\nonlin(\zvel)
&=-\curl\zvelss\wedge\zvelss-\proj\np{\curl\zvelpp\wedge\zvelpp}, 
\label{eq:nonlinCurl_ss}\\
\nonlinpp(\zvel)
\coloneqq\projcompl\nonlin(\zvel)
&=-\curl\zvelss\wedge\zvelpp-\curl\zvelpp\wedge\zvelss
-\projcompl\np{\curl\zvelpp\wedge\zvelpp}, 
\label{eq:nonlinCurl_pp}
\end{align}
with $\zvelss\coloneqq\proj\zvel$ and $\zvelpp\coloneqq\projcompl\zvel$.
For $(t,x)\in\torus\times\ball^\radius$ from Proposition \ref{prop:RepresentationVorticity} we then obtain
\begin{align}
D_x^\alpha\proj\calf_\radius(\zvel)(x)
&=D_x^\alpha\fundsolvelss\ast\bb{\np{1-\cutoff_\radius}
\nonlinss(\zvel)}(x), 
\label{eq:Representation_Fss}\\
D_x^\alpha\projcompl\calf_\radius(\zvel)(t,x)
&=D_x^\alpha\fundsolvelpp\ast\bb{\np{1-\cutoff_\radius}\nonlinpp(\zvel)}(t,x), 
\label{eq:Representation_Fpp}\\ 
\curl\proj\calf_\radius(\zvel)(x)
&=\int_{\R^3} \grad\fundsolvortss(x-y) \wedge 
\bb{\np{1-\cutoff_\radius}\nonlinss(\zvel)}(y)\,\dy,
\label{eq:Representation_CurlFss}\\
\curl\projcompl\calf_\radius(\zvel)(t,x)
&=\int_{\torus\times\R^3} \grad\fundsolvortpp(t-s,x-y)
\wedge\bb{\np{1-\cutoff_\radius}\nonlinpp(\zvel)}(s,y)\,\dsy,
\label{eq:Representation_CurlFpp}
\end{align}
and
\begin{align}
D_x^\alpha\proj\calh_\radius(x)
&=D_x^\alpha\fundsolvelss\ast\bb{\proj f+\cutoff_\radius\nonlinss(\uvel)}(x), 
\label{eq:Representation_Hss}\\
D_x^\alpha\projcompl\calh_\radius(t,x)
&=D_x^\alpha\fundsolvelpp\ast\bb{\projcompl f+\cutoff_\radius\nonlinpp(\uvel)}(t,x), 
\label{eq:Representation_Hpp}\\ 
\curl\proj\calh_\radius(x)
&=\int_{\R^3} \grad\fundsolvortss(x-y) \wedge 
\bb{\proj f+\cutoff_\radius\nonlinss(\uvel)}(y)\,\dy,
\label{eq:Representation_CurlHss}\\
\curl\projcompl\calh_\radius(t,x)
&=\int_{\torus\times\R^3} \grad\fundsolvortpp(t-s,x-y) 
\wedge\bb{\projcompl f+\cutoff_\radius\nonlinpp(\uvel)}(s,y)\,\dsy.
\label{eq:Representation_CurlHpp}
\end{align}

In the next step we introduce the functional framework
for the analysis of the fixed-point equation \eqref{eq:uFixedPoint_VortDecay}.
Let $\varepsilon\in\np{0,\frac{1}{4}}$ and fix a radius $\radius>\radius_0$.
We define the following (semi-)norms,
which take into account different decay rates of the steady-state and the purely periodic parts:
\begin{align*}
\velnorm{\zvel}
&\coloneqq
\esssup_{x\in\ball^{\radius}}\Bb{\snorm{x}\np{1+\wakefct{x}}\snorm{\proj\zvel(x)}
+\bb{\snorm{x}\np{1+\wakefct{x}}}^{3/2}\snorm{\grad\proj\zvel(x)}}
\\
&\qquad\qquad\qquad\qquad
+\esssup_{\np{t,x}\in\torus\times\ball^{\radius}}\Bb{\snorm{x}^3\snorm{\projcompl\zvel(t,x)}
+\snorm{x}^4\snorm{\grad\projcompl\zvel(t,x)}}, 
\\
\vortnorm{\zvel}
&\coloneqq\esssup_{x\in\ball^{\radius}}\snorm{x}^{3/2}
\e^{\frac{\wakefct{\Kconst x}}{1+\radius}}\snorm{\curl\proj\zvel(x)} 
\\
&\qquad\qquad\qquad\qquad
+\esssup_{\np{t,x}\in\torus\times\ball^{\radius}}\snorm{x}^{9/2-\varepsilon}
\e^{\frac{\wakefct{\Kconst x}}{1+ \radius}}\snorm{\curl\projcompl\zvel(t,x)},
\end{align*}
where $\Kconst\coloneqq\frac{1}{4}\min\set{\rey, \Cexpfs}$
with $ \Cexpfs$ from Theorem \ref{thm:FundsolVortpp_est}.
The function spaces associated to these (semi-)norms are given by
\begin{align*}
\velspace&\coloneqq\setcl{\zvel\in\WSRloc{1}{1}(\torus\times\ball^\radius)}{\velnorm{\zvel}<\infty},
\\
\vortspace&\coloneqq\setcl{\zvel\in\velspace}{\vortnorm{\zvel}<\infty},
\end{align*}
which are Banach spaces with respect to the norms
\[
\norm{\zvel}_{\velspace}\coloneqq\velnorm{\zvel}, 
\qquad
\norm{\zvel}_{\vortspace}\coloneqq\velnorm{\zvel}+\vortnorm{\zvel},
\]
respectively.

\begin{rem}
Let us explain the terms appearing in these definitions.
The definition of $\velnorm{\zvel}$ is chosen to capture the asymptotic behavior of $\uvel$ and $\grad\uvel$
described in Theorem \ref{thm:VelocityDecay}.
A justification for the denominator $1+\radius$ 
in the exponential term in the definition of $\vortnorm{\zvel}$ 
is given by 
Lemma \ref{lem:Hs_EstExpFct} below. 
The choice of the constant $\Kconst$ ensures the validity of the inequalities
\begin{align}\label{est:Kconst}
\e^{2\wakefct{\Kconst x}}
\leq \e^{\wakefct{\rey x}/2}, \qquad\qquad
\e^{2\wakefct{\Kconst x}}
\leq \e^{ \Cexpfs\snorm{x}},
\end{align}
so that the exponential term can be related with the exponential terms 
in the decay rates of $\grad\fundsolvortss$ and $\grad\fundsolvortpp$
from Theorem \ref{thm:FundsolVortss_est} and Theorem \ref{thm:FundsolVortpp_est}, respectively.
Moreover, the second term in the definition of $\vortnorm{\zvel}$
contains the factor $\snorm{x}^{9/2-\varepsilon}$ instead of $\snorm{x}^{9/2}$,
which one would expect, in view of the asserted estimate \eqref{est:VorticityDecay_pp}.
Later on we shall see that this discrepancy is necessary to ensure that $\calf_\radius$ 
is a contraction in the underlying function space. 
\end{rem}

\section{Estimates}
\label{sec:Estimates}

In this section, we collect estimates of $\calh_\radius$ and $\calf_\radius(z)$
with respect to the (semi-)norms introduced above,
which ensure that $\zvel\mapsto\calf_\radius(\zvel)+\calh_\radius$
is a contractive self-mapping when we choose $\radius$ sufficiently large.
We begin with the following elementary lemma,
which explains the term $1+\radius$
in the definition of $\vortnorm{\zvel}$.

\begin{lem}\label{lem:Hs_EstExpFct}
Let $a,\,\radius>0$.
If $x,y\in\R^3$ with $\snorm{y}\leq2\radius$, then 
\begin{align}
\e^{-\wakefct{a\np{x-y}}}
&\leq\e^{4a}\e^{-\frac{\wakefct{a x}}{1+\radius}},
\label{est:Hs_vortest_ExpWakefct}
\\
\e^{-a\snorm{x-y}}
&\leq\e^{2a}\e^{-\frac{a\snorm{x}}{1+\radius}}.
\label{est:Hs_vortest_ExpAbs}
\end{align}
\end{lem}

\begin{proof}
For $\snorm{y}\leq2\radius$ we have
$\wakefct{a y}/\np{1+\radius}\leq 2a\snorm{y}/\np{1+\radius}\leq4a$.
Together with $\wakefct{a\np{x-y}}\geq\wakefct{ax}-\wakefct{ay}$, this implies
\[
\e^{-\wakefct{a\np{x-y}}}\leq\e^{-\frac{\wakefct{a\np{x-y}}}{1+\radius}}
\leq\e^{-\frac{\wakefct{a x}}{1+\radius}}\e^{\frac{\wakefct{a y}}{1+\radius}}
\leq\e^{-\frac{\wakefct{a x}}{1+\radius}}\e^{4a}.
\]
Similarly, we have
$\snorm{y}/(1+\radius)\leq 2$, 
which implies
\[
\e^{-a\snorm{x-y}}\leq\e^{-\frac{a\snorm{x-y}}{1+\radius}}
\leq\e^{-\frac{a\snorm{x}}{1+\radius}}\e^{\frac{a\snorm{y}}{1+\radius}}
\leq\e^{-\frac{a\snorm{x}}{1+\radius}}\e^{2a}.
\]
This completes the proof.
\end{proof}

We further employ the following lemma in order to estimate convolutions 
of functions with anisotropic decay behavior.

\begin{lem}\label{lem:ConvEst_WakeFct_ThreeHalfs}
Let $A\in(2,\infty)$, $B\in[0,\infty)$ with $A+\min\set{1,B}>3$.
Then there exists $C=C\np{A,B}>0$ such that 
for all $x\in\R^3$ it holds
\[
\int_{\R^3}\bb{\np{1+\snorm{x-y}}\np{1+\wakefct{x-y}}}^{-3/2}
\np{1+\snorm{y}}^{-A}\np{1+\wakefct{y}}^{-B}\,\dy 
\leq C\np{1+\snorm{x}}^{-3/2}.
\]
\end{lem}

\begin{proof}
See \cite[Theorem 5]{DeuringGaldi_ExpDecayVorticity_2016}.
\end{proof}

The next lemma
treats convolutions of functions that are homogeneous in space.

\begin{lem}\label{lem:ConvolutionExp}
Let $A\in(0,3)$, $B\in(0,\infty)$, $\alpha\in(0,\infty)$. 
Then there exists a constant $C=C\np{A,B,\alpha}>0$
such that for all $x\in\R^3$ it holds
\[
\int_{\R^3}\snorm{x-y}^{-A}
\e^{-\alpha\snorm{x-y}}\np{1+\snorm{y}}^{-B}\,\dy
\leq C\np{1+\snorm{x}}^{-B}.
\]
\end{lem}

\begin{proof}
For $x=0$ the integral is finite,
so that it remains to consider $x\neq0$.
We split the integral into two parts
\begin{align*}
I_1
&\coloneqq\int_{\ball_{\snorm{x}/2}(x)}\snorm{x-y}^{-A}
\e^{-\alpha\snorm{x-y}}\np{1+\snorm{y}}^{-B}\,\dy,\\
I_2
&\coloneqq\int_{\ball^{\snorm{x}/2}(x)}\snorm{x-y}^{-A}
\e^{-\alpha\snorm{x-y}}\np{1+\snorm{y}}^{-B}\,\dy,
\end{align*}
which we estimate separately.
On the one hand, 
since $\snorm{x-y}\leq\snorm{x}/2$ implies $\snorm{y}\geq\snorm{x}-\snorm{x-y}\geq\snorm{x}/2$,
we have
\[
I_1
\leq C\np{1+\snorm{x}}^{-B}\int_{\R^3}\snorm{x-y}^{-A}\e^{-\alpha\snorm{x-y}}\,\dy
\leq C\np{1+\snorm{x}}^{-B},
\]
where the integral is finite due to $A<3$.
On the other hand, we obtain 
\[
I_2
\leq C\e^{-\alpha\snorm{x}/4}
\int_{\R^3}\e^{-\alpha\snorm{x-y}/2}\,\dy
\leq C\e^{-\alpha\snorm{x}/4}
\leq C\np{1+\snorm{x}}^{-B}.
\]
This completes the proof.
\end{proof}

Since our assumptions do not provide pointwise information
on the body force $f$,
we estimate the convolutions of the fundamental 
solutions with $f$ in a different way,
which leads to the following lemma.

\begin{lem}\label{lem:Conv_fLp}
There exists a constant $C>0$ such that 
for $\alpha\in\N_0$, $\snorm{\alpha}\leq1$,
we have
\begin{align}
\snorml{D_x^\alpha\fundsolvelss\ast{\proj f}(x)}
&\leq C\bb{\snorm{x}\bp{1+\wakefct{\rey x}}}^{-1-\frac{\snorm{\alpha}}{2}}, 
\label{est:Conv_fLp_velss}
\\
\snorml{D_x^\alpha\fundsolvelpp\ast{\projcompl f}(t,x)}
&\leq C\snorm{x}^{-3-\snorm{\alpha}}, 
\label{est:Conv_fLp_velpp}
\\
\snorml{\int_{\R^3} \grad\fundsolvortss(x-y) \wedge {\proj f}(y)\,\dy}
&\leq C\snorm{x}^{-3/2}\e^{-\frac{\wakefct{\rey x}}{4\np{1+\radius_0}}},
\label{est:Conv_fLp_vortss}
\\
\snorml{\int_{\torus\times\R^3} \grad\fundsolvortpp(t-s,x-y) 
\wedge{\projcompl f}(s,y)\,\dsy}
&\leq C\snorm{x}^{-9/2}\e^{-\frac{\Cexpfs\snorm{x}}{2\np{1+\radius_0}}}.
\label{est:Conv_fLp_vortpp}
\end{align}
for all $t\in\torus$ and $\snorm{x}\geq 2\radius_0$.
\end{lem}

\begin{proof}
For $\snorm{x}\geq 2\radius_0\geq 2\snorm{y}$ 
we have
\[
(1+2\rey\radius_0)\np{1+\rey\wakefct{x-y}}
\geq 1+\rey\wakefct{x}+2\rey\radius_0-\rey\wakefct{y}
\geq 1+\rey\wakefct{x}
\]
and $\snorm{x-y}\geq\snorm{x}-\snorm{y}\geq\snorm{x}/2\geq \radius_0$.
Therefore, \eqref{est:fundsolss_Decay} and 
$\supp f \subset \torus\times\ball_{\radius_0}$
imply
\[
\begin{aligned}
\snorml{D_x^\alpha\fundsolvelss\ast{\proj f}(x)}
&\leq C\int_{\ball_{\radius_0}}
\bb{\snorm{x-y}\bp{1+\rey\wakefct{x-y}}}^{-1-\frac{\snorm{\alpha}}{2}}
\snorm{\proj f(y)}\,\dy
\\
&\leq C 
\bb{\snorm{x}\bp{1+\rey\wakefct{x}}}^{-1-\frac{\snorm{\alpha}}{2}}
\int_{\ball_{\radius_0}} \snorm{\proj f(y)}\,\dy,
\end{aligned}
\]
which yields \eqref{est:Conv_fLp_velss}.
Using H\"older's inequality in time
and \eqref{est:tpfundsol_ComplPointwiseEst}, for any $q\in(1,\infty)$
we obtain in a similar way
\[
\begin{aligned}
\snorml{D_x^\alpha\fundsolvelpp\ast{\projcompl f}(t,x)}
&\leq \int_{\ball_{\radius_0}}
\Bp{\int_\torus \snorm{D_x^\alpha\fundsolvelpp(s,x-y)}^{\frac{q}{q-1}}\,\ds}^{\frac{q-1}{q}}
\Bp{\int_\torus\snorm{\projcompl f(s,y)}^q\,\ds}^{1/q}\,\dy
\\
&\leq C\int_{\ball_{\radius_0}}
\snorm{x-y}^{-3-\snorm{\alpha}}
\Bp{\int_\torus\snorm{\projcompl f(s,y)}^q\,\ds}^{1/q}\,\dy
\\
&\leq C 
\snorm{x}^{-3-\snorm{\alpha}}
\Bp{\int_\torus\int_{\ball_{\radius_0}}\snorm{\projcompl f(s,y)}^q\,\ds\dy}^{1/q},
\end{aligned}
\]
which shows \eqref{est:Conv_fLp_velpp}.
In virtue of the estimates \eqref{est:FundsolVortss_grad}
and \eqref{est:FundsolVortpp_grad} and
Lemma \ref{lem:Hs_EstExpFct},
for $\snorm{x}\geq2\radius_0\geq2\snorm{y}$ we further derive
\[
\begin{aligned}
\snorm{\grad\fundsolvortss(x-y)}
&\leq C\snorm{x-y}^{-3/2}
\e^{-\wakefct{\rey\np{x-y}}/4} 
\leq C \snorm{x}^{-3/2}\e^{-\frac{\wakefct{\rey x}}{4\np{1+\radius_0}}}, 
\\
\norm{\grad\fundsolvortpp(\cdot,x-y)}_{\LR{q}(\torus)}
&\leq C\snorm{x-y}^{-9/2}\e^{- \Cexpfs\snorm{x-y}/2}
\leq C \snorm{x}^{-9/2}\e^{-\frac{-\Cexpfs\snorm{x}}{2\np{1+\radius_0}}}.
\end{aligned}
\]
From these estimates 
we conclude \eqref{est:Conv_fLp_vortss} and \eqref{est:Conv_fLp_vortpp}
with the same argument as above.
\end{proof}

After these preparations,
we show in the next two lemmas that the norm of $\calh_\radius$ 
in both $\velspace$ and $\vortspace$ is bounded by a constant 
independent of $\radius\geq 2\radius_0$.

\begin{lem}\label{lem:Hs_velest}
There exists a constant $ \CHsvelnorm>0$ such that for all $\radius\in[2\radius_0,\infty)$ 
we have
\[
\velnorm{\calh_\radius}\leq \CHsvelnorm.
\]
\end{lem}

\begin{proof}
From the decay estimates of $\uvel$ and $\grad\uvel$ from 
Theorem \ref{thm:VelocityDecay} 
we conclude
\begin{align}
\snorml{\cutoff_\radius(x) \nonlinss\np{\uvel}(x)}
&\leq C \bb{\np{1+\snorm{x}}\np{1+\wakefct{x}}}^{-5/2},
\label{est:Hs_rhsss}\\
\snorml{\cutoff_\radius(x) \nonlinpp\np{\uvel}(t,x)}
&\leq C \np{1+\snorm{x}}^{-9/2}.
\label{est:Hs_rhspp}
\end{align}
By Theorem \ref{thm:ConvFundsolss} and Theorem \ref{thm:ConvFundsolpp},
these estimates and the formulas \eqref{eq:Representation_Hss} 
and \eqref{eq:Representation_Hpp}
together with Lemma \ref{lem:Conv_fLp}
imply
\[
\begin{aligned}
\snorm{\proj\calh_\radius(x)}
&\leq C \bb{\np{1+\snorm{x}}\np{1+\wakefct{x}}}^{-1},
\\
\snorm{\grad\proj\calh_\radius(x)}
&\leq C \bb{\np{1+\snorm{x}}\np{1+\wakefct{x}}}^{-3/2},
\\
\snorm{\projcompl\calh_\radius(t,x)}
&\leq C \np{1+\snorm{x}}^{-3},
\\
\snorm{\grad\projcompl\calh_\radius(t,x)}
&\leq C \np{1+\snorm{x}}^{-4}
\end{aligned}
\]
for all $t\in\torus$ and $\snorm{x}\geq\radius_0$.
Collecting these, we arrive at the claimed estimate.
\end{proof}

\begin{lem}\label{lem:Hs_vortest}
There exists a constant $ \CHsvortnorm>0$ such that for all $\radius\in[2\radius_0,\infty)$ we have
\[
\vortnorm{\calh_\radius}\leq \CHsvortnorm.
\]
\end{lem}

\begin{proof}
At first, let $x\in\R^3$ with $\snorm{x}\geq2\radius$.
For $\snorm{y}\leq7\radius/4$ we have
\[
\snorm{x-y}\geq\snorm{x}-\snorm{y}\geq\snorm{x}-7\radius/4
\geq\snorm{x}-7\snorm{x}/8
=\snorm{x}/8
\geq\radius/4\geq\radius_0/2.
\]
From \eqref{est:FundsolVortss_grad} and Lemma \ref{lem:Hs_EstExpFct}, we then conclude
\begin{align*}
\snorm{\grad\fundsolvortss(x-y)}
&\leq C \bp{\snorm{x-y}^{-2}+\snorm{x-y}^{-3/2}\wakefct{\rey \np{x-y}}^{1/2}}
\e^{-\wakefct{\rey \np{x-y}}/2}\\
&\leq C \bp{1+\snorm{x-y}^{-3/2}\bp{1+\wakefct{\rey\np{x-y}}}^{-3/2}}
\e^{-\wakefct{\rey \np{x-y}}/4}\\
&\leq C \bb{\np{1+\snorm{x-y}}\bp{1+\wakefct{\rey\np{x-y}}}}^{-3/2}
\e^{-\frac{\wakefct{\rey x}}{4\np{1+\radius}}}.
\end{align*}
In virtue of 
\eqref{eq:Representation_CurlHss}, \eqref{est:Conv_fLp_vortss} 
and \eqref{est:Hs_rhsss}
we thus obtain
\begin{align*}
&\snorm{\curl\proj\calh_\radius(x)}
\leq C\int_{\ball_{7\radius/4}} \snorml{\grad\fundsolvortss(x-y)} \,
\snorml{\proj f+\cutoff_\radius\nonlinss(\uvel)}(y)\,\dy\\
&\leq C\snorm{x}^{-3/2}\e^{-\frac{\wakefct{\rey x}}{4\np{1+\radius_0}}} \\
&\qquad + C \e^{-\frac{\wakefct{\rey x}}{4\np{1+\radius}}}
\int_{\R^3}\bb{\bp{1+\snorm{x-y}}\bp{1+\wakefct{x-y}}}^{-3/2}
\bb{\np{1+\snorm{y}}\np{1+\wakefct{y}}}^{-5/2}\,\dy
\end{align*}
for $\snorm{x}\geq2\radius\geq4\radius_0$.
By estimating the remaining integral with the help of Lemma \ref{lem:ConvEst_WakeFct_ThreeHalfs} 
and employing \eqref{est:Kconst}, we deduce 
\begin{equation}\label{est:Hs_vortest_ss}
\snorm{\curl\proj\calh_\radius(x)}
\leq C \e^{-\frac{\wakefct{\Kconst x}}{1+\radius}}\snorm{x}^{-3/2}
\end{equation}
for $\snorm{x}\geq2\radius$.
If $\radius\leq\snorm{x}\leq2\radius$,
then Lemma \ref{lem:Hs_velest} yields
\[
\snorm{\curl\proj\calh_\radius(x)}
\leq C \snorm{\grad\proj\calh_\radius(x)}
\leq C \bb{\np{1+\snorm{x}}\np{1+\wakefct{x}}}^{-3/2}
\leq C \snorm{x}^{-3/2}.
\]
Since $\snorm{x}\leq2\radius$ implies  
$
\wakefct{\Kconst x}/\np{1+\radius}
\leq 2\snorm{\Kconst x}/\np{1+\radius}
\leq 4\Kconst\radius/\np{1+\radius}
\leq 4\Kconst,
$
we have 
$1\leq\e^{4\Kconst}\e^{-\wakefct{\Kconst x}/\np{1+\radius}}$,
so that \eqref{est:Hs_vortest_ss} also holds for $\radius\leq\snorm{x}\leq2\radius$. 

Now let us turn to $\curl\projcompl\calh_\radius$. 
From \eqref{est:FundsolVortpp_grad}
and \eqref{est:Hs_vortest_ExpAbs}, for $\snorm{y}\leq2\radius$
we conclude
\[
\int_\torus\snorm{\grad\fundsolvortpp(t-s,x-y)}\,\ds
\leq C \snorm{x-y}^{-5/2} \e^{-\frac{ \Cexpfs\snorm{x-y}}{2}} 
\e^{-\frac{ \Cexpfs\snorm{x}}{2\np{1+\radius}}},
\]
so that  \eqref{eq:Representation_CurlHpp}, \eqref{est:Conv_fLp_vortpp}
and \eqref{est:Hs_rhspp} lead to
\begin{align*}
&\snorm{\curl\projcompl\calh_\radius(t,x)}
\leq C \int_{\ball_{7\radius/4}}\int_\torus\snorml{\grad\fundsolvortpp(t-s,x-y)}
\,\snorml{\projcompl f+\cutoff_\radius\nonlinpp\np{\uvel}}(s,y) \,\ds\dy\\
&\quad\qquad
\leq C\snorm{x}^{-9/2}\e^{-\frac{\Cexpfs\snorm{x}}{2\np{1+\radius_0}}}
+C \e^{-\frac{ \Cexpfs\snorm{x}}{2\np{1+\radius}}}
\int_{\R^3}\snorm{x-y}^{-5/2}\e^{- \Cexpfs\snorm{x-y}/2}
\np{1+\snorm{y}}^{-9/2}\,\dy.
\end{align*}
The remaining integral can be estimated with Lemma \ref{lem:ConvolutionExp}.
Further using \eqref{est:Kconst}, we end up with
\[
\snorm{\curl\projcompl\calh_\radius(t,x)}
\leq C \e^{-\frac{ \Cexpfs\snorm{x}}{2\np{1+\radius}}}
\snorm{x}^{-9/2}
\leq C \e^{-\frac{\wakefct{\Kconst x}}{1+\radius}}
\snorm{x}^{-9/2+\varepsilon}
\]
for $\snorm{x}\geq\radius\geq2\radius_0$ and $t\in\torus$.
A combination of this estimate with \eqref{est:Hs_vortest_ss} finishes the proof.
\end{proof}

In the next two lemmas we provide appropriate estimates of $\calf_\radius(\zvel)$.
Observe that, in contrast to  $\calh_\radius$, this term depends on the (unknown) function
$\zvel$. 
In order to eventually obtain a contraction for large $\radius$, 
we factor out the term $\radius^{-\varepsilon}$ in the estimates.

\begin{lem}\label{lem:Fs_velest}
There exists a constant $ \CFsvelnorm>0$ such that for all $\radius\in[2\radius_0,\infty)$
and all $\zvel_1,\zvel_2\in\velspace$ we have
\begin{align}
\velnorm{\calf_\radius(\zvel_1)}
&\leq \CFsvelnorm\radius^{-\varepsilon}\velnorm{\zvel_1}^2,
\label{est:Fs_quadratic}
\\
\velnorm{\calf_\radius(\zvel_1)-\calf_\radius(\zvel_2)}
&\leq \CFsvelnorm\radius^{-\varepsilon}\bp{\velnorm{\zvel_1}+\velnorm{\zvel_2}}
\velnorm{\zvel_1-\zvel_2}.
\label{est:Fs_contraction}
\end{align}
\end{lem}

\begin{proof}
For $\zvel\in\velspace$ 
we immediately deduce
\begin{align*}
\snorml{\np{1-\cutoff_\radius(x)}\nonlinss(\zvel)(x)}
&\leq C \velnorm{\zvel}^2\np{1-\cutoff_\radius(x)}\bb{\np{1+\snorm{x}}\np{1+\wakefct{x}}}^{-5/2} \\
&\leq C \radius^{-\varepsilon}\velnorm{\zvel}^2\np{1+\snorm{x}}^{-5/2+\varepsilon}\np{1+\wakefct{x}}^{-5/2},
\\
\snorml{\np{1-\cutoff_\radius(x)} \nonlinpp\np{\zvel}(t,x)}
&\leq C \velnorm{\zvel}^2\np{1-\cutoff_\radius(x)}\np{1+\snorm{x}}^{-9/2}\\
&\leq C \radius^{-\varepsilon}\velnorm{\zvel}^2\np{1+\snorm{x}}^{-9/2+\varepsilon}
\end{align*}
for $\snorm{x}\geq\radius$.
By Theorem \ref{thm:ConvFundsolss} and Theorem \ref{thm:ConvFundsolpp},
from these estimates and the formulas \eqref{eq:Representation_Fss}
and \eqref{eq:Representation_Fpp}
we conclude
\[
\begin{aligned}
\snorm{\proj\calf_\radius(\zvel)(x)}
&\leq C \radius^{-\varepsilon}\velnorm{\zvel}^2\bb{\np{1+\snorm{x}}\np{1+\wakefct{x}}}^{-1},
\\
\snorm{\grad\proj\calf_\radius(\zvel)(x)}
&\leq C \radius^{-\varepsilon}\velnorm{\zvel}^2\bb{\np{1+\snorm{x}}\np{1+\wakefct{x}}}^{-3/2},
\\
\snorm{\projcompl\calf_\radius(\zvel)(t,x)}
&\leq C \radius^{-\varepsilon}\velnorm{\zvel}^2\np{1+\snorm{x}}^{-3},
\\
\snorm{\grad\projcompl\calf_\radius(\zvel)(t,x)}
&\leq C \radius^{-\varepsilon}\velnorm{\zvel}^2\np{1+\snorm{x}}^{-4}.
\end{aligned}
\]
Collecting these estimates, we obtain \eqref{est:Fs_quadratic}.
 The inequality \eqref{est:Fs_contraction} is proved in the same fashion.
\end{proof}

\begin{lem}\label{lem:Fs_vortest}
There exists a constant $ \CFsvortnorm>0$ such that for all $\radius\in[2\radius_0,\infty)$
and all $\zvel_1,\zvel_2\in\vortspace$ we have
\begin{align}
\vortnorm{\calf_\radius(\zvel_1)}&\leq \CFsvortnorm\radius^{-\varepsilon}\velnorm{\zvel_1}\vortnorm{\zvel_1}, 
\label{est:Fs_curlquadratic}\\
\vortnorm{\calf_\radius(\zvel_1)-\calf_\radius(\zvel_2)}
&\leq \CFsvortnorm\radius^{-\varepsilon}
\bp{\norm{\zvel_1}_{\vortspace}+\norm{\zvel_2}_{\vortspace}}
\norm{\zvel_1-\zvel_2}_{\vortspace}.
\label{est:Fs_curlcontraction}
\end{align}
\end{lem}

\begin{proof}
For $\zvel\in\vortspace$ we have
\begin{align}
&\begin{aligned}
\snorml{\np{1-\cutoff_\radius(x)}\nonlinss(\zvel)(x)}
&\leq C \velnorm{\zvel}\vortnorm{\zvel}\np{1-\cutoff_\radius(x)}
\snorm{x}^{-5/2}\np{1+\wakefct{x}}^{-1}\e^{-\frac{\wakefct{\Kconst x}}{1+\radius}}  \\
&\leq C \radius^{-\varepsilon}\velnorm{\zvel}\vortnorm{\zvel}
\snorm{x}^{-5/2+\varepsilon}\np{1+\wakefct{x}}^{-1}\e^{-\frac{\wakefct{\Kconst x}}{1+\radius}},
\end{aligned}
\label{est:FsCurl_rhsss}
\\
&\begin{aligned}
\snorml{\np{1-\cutoff_\radius(x)} \nonlinpp\np{\zvel}(t,x)}
&\leq C \velnorm{\zvel}\vortnorm{\zvel}\np{1-\cutoff_\radius(x)}\snorm{x}^{-9/2}
\e^{-\frac{\wakefct{\Kconst x}}{1+\radius}}\\
&\leq C \radius^{-\varepsilon}\velnorm{\zvel}\vortnorm{\zvel}\snorm{x}^{-9/2+\varepsilon}
\e^{-\frac{\wakefct{\Kconst x}}{1+\radius}}
\end{aligned}
\label{est:FsCurl_rhspp}
\end{align}
for $\snorm{x}\geq\radius$.
Exploiting the representation formula \eqref{eq:Representation_CurlFss},
we can employ \eqref{est:FundsolVortss_grad} and \eqref{est:FsCurl_rhsss} to estimate
\begin{align*}
\snorm{\curl\proj\calf_\radius(\zvel)(x)}
&\leq C\int_{\R^3}\snorml{\grad\fundsolvortss(x-y)}\,\snorml{
\np{1-\cutoff_\radius(y)}\nonlinss(\zvel)(y)}\,\dy\\
&\leq C \radius^{-\varepsilon}\velnorm{\zvel}\vortnorm{\zvel}\bp{I_1+I_2},
\end{align*}
where
\[
\begin{aligned}
I_1
&\coloneqq\int_{\ball^\radius\cap\ball_{\radius_0}(x)}
\snorm{x-y}^{-2}\e^{-\frac{\wakefct{\rey\np{x-y}}}{4}}
\snorm{y}^{-5/2+\varepsilon}\np{1+\wakefct{y}}^{-1}\e^{-\frac{\wakefct{\Kconst y}}{1+\radius}}\,\dy, 
\\
I_2
&\coloneqq\int_{\ball^\radius\cap\ball^{\radius_0}(x)}
\bb{\snorm{x-y}\wakefct{\rey \np{x-y}}}^{-3/2}\e^{-\frac{\wakefct{\rey\np{x-y}}}{4}} \snorm{y}^{-5/2+\varepsilon}\np{1+\wakefct{y}}^{-1}
\e^{-\frac{\wakefct{\Kconst y}}{1+\radius}}\,\dy.
\end{aligned}
\]
To give estimates of these integrals, we first note that by
$\wakefct{\rey(x-y)}\geq\wakefct{\rey x}-\wakefct{\rey y}$
and \eqref{est:Kconst}, we have
\begin{align}\label{est:Fs_vortest_ExpConv}
\e^{-\frac{\wakefct{\rey\np{x-y}}}{4}}\e^{-\frac{\wakefct{\Kconst y}}{1+\radius}}
\leq\e^{-\frac{\wakefct{\rey x}}{4\np{1+\radius}}}
\e^{\frac{\wakefct{\rey y}}{4\np{1+\radius}}}
\e^{-\frac{\wakefct{\Kconst y}}{1+\radius}}
\leq\e^{-\frac{\wakefct{\Kconst x}}{1+\radius}}
\end{align}
for all $x,\,y\in\R^3$.
On the one hand, exploiting this estimate and
that $\snorm{x-y}\leq\radius_0\leq\snorm{x}/2$ implies
$\snorm{y}\geq\snorm{x}-\snorm{x-y}
\geq\snorm{x}-\radius_0
\geq\snorm{x}/2$,
we conclude
\[
I_1\leq C \e^{-\frac{\wakefct{\Kconst x}}{1+\radius}}\snorm{x}^{-5/2+\varepsilon}
\int_{\ball_{\radius_0}(x)}\snorm{x-y}^{-2}\,\dy
\leq C \e^{-\frac{\wakefct{\Kconst x}}{1+\radius}}\snorm{x}^{-3/2}
\]
for $\snorm{x}\geq\radius\geq2\radius_0$.
On the other hand, due to \eqref{est:Fs_vortest_ExpConv}
and the fact that $\snorm{y}\geq\radius\geq2\radius_0$ implies 
$\snorm{y}\geq C \np{1+\snorm{y}}$, we obtain
\begin{align*}
I_2
&\leq C 
\e^{-\frac{\wakefct{\Kconst x}}{1+\radius}}
\int_{\R^3}\bb{\np{1+\snorm{x-y}}\wakefct{x-y}}^{-3/2}
\np{1+\snorm{y}}^{-5/2+\varepsilon}\np{1+\wakefct{y}}^{-1}\,\dy\\
&\leq C \e^{-\frac{\wakefct{\Kconst x}}{1+\radius}}\snorm{x}^{-3/2}
\end{align*}
by Lemma \ref{lem:ConvEst_WakeFct_ThreeHalfs}.
From the estimates of $I_1$ and $I_2$
we deduce
\[
\snorm{\curl\proj\calf_\radius(\zvel)(x)}
\leq C \radius^{-\varepsilon}\velnorm{\zvel}\vortnorm{\zvel}
\e^{-\frac{\wakefct{\Kconst x}}{1+\radius}}
\snorm{x}^{-3/2}.
\]
Now let us turn to the purely periodic part $\projcompl\calf_\radius(\zvel)$.
From \eqref{est:FundsolVortpp_grad} (with $q=1$ and $\gamma=1/4$) we conclude
\[
\int_\torus\snorm{\grad\fundsolvortpp(t-s,x-y)}\,\ds
\leq C \snorm{x-y}^{-5/2} \e^{- \Cexpfs\snorm{x-y}}.
\]
With formula \eqref{eq:Representation_CurlFpp} and estimate \eqref{est:FsCurl_rhspp}
we thus obtain
\begin{align*}
\snorm{\curl\projcompl\calf_\radius(\zvel)(t,x)}
&\leq C\int_\torus\int_{\R^3}\snorml{\grad\fundsolvortpp(t-s,x-y)}\,\snorml{
\np{1-\cutoff_\radius(y)}\nonlinpp(\zvel)(s,y)}\,\dy\ds\\
&\leq C \radius^{-\varepsilon}\velnorm{\zvel}\vortnorm{\zvel}
\int_{\ball^\radius}\snorm{x-y}^{-5/2} \e^{- \Cexpfs\snorm{x-y}}
\snorm{y}^{-9/2+\varepsilon}\e^{-\frac{\wakefct{\Kconst y}}{1+\radius}}\,\dy.
\end{align*}
By \eqref{est:Kconst} we have
\[
\e^{-\frac{ \Cexpfs\snorm{x-y}}{2}}\e^{-\frac{\wakefct{\Kconst y}}{1+\radius}}
\leq\e^{-\wakefct{\Kconst\np{x-y}}}\e^{-\frac{\wakefct{\Kconst y}}{1+\radius}}
\leq\e^{-\frac{\wakefct{\Kconst\np{x-y}}}{1+\radius}}\e^{-\frac{\wakefct{\Kconst y}}{1+\radius}}
\leq\e^{-\frac{\wakefct{\Kconst x}}{1+\radius}}.
\]
This yields
\[
\begin{aligned}
&\snorm{\curl\projcompl\calf_\radius(\zvel)(t,x)}\\
&\qquad\leq C \radius^{-\varepsilon}\velnorm{\zvel}\vortnorm{\zvel}
\e^{-\frac{\wakefct{\Kconst x}}{1+\radius}}
\int_{\R^3}\snorm{x-y}^{-5/2} \e^{-\frac{ \Cexpfs\snorm{x-y}}{2}}
\np{1+\snorm{y}}^{-9/2+\varepsilon}\,\dy.
\end{aligned}
\]
Employing Lemma \ref{lem:ConvolutionExp} to estimate the remaining integral,
we end up with
\[
\snorm{\curl\projcompl\calf_\radius(\zvel)(t,x)}
\leq C \radius^{-\varepsilon}\velnorm{\zvel}\vortnorm{\zvel}
\e^{-\frac{\wakefct{\Kconst x}}{1+\radius}}\snorm{x}^{-9/2+\varepsilon}.
\]
In total, we have thus shown \eqref{est:Fs_curlquadratic}. 
Estimate \eqref{est:Fs_curlcontraction} is derived in the same way.
\end{proof}

\section{Conclusion of the proof}
\label{sec:ConclusionProof}

After the preparatory results from the previous section, 
we now prove the existence of a function $\zvel\in\vortspace$
satisfying the fixed-point equation
\[
\zvel=\calf_\radius(\zvel)+\calh_\radius
\]
provided $\radius\geq2\radius_0$ is chosen sufficiently large.
Afterwards, we show uniqueness of this fixed point in the function class $\velspace$.
Since $\restriction{\uvel}{\torus\times\ball^\radius}$ is another solution to
this fixed-point equation and belongs to $\velspace$ by Theorem \ref{thm:VelocityDecay},
we then conclude that $\zvel$ coincides with $\restriction{\uvel}{\torus\times\ball^\radius}$. 
This yields the decay rate of the vorticity field asserted in Theorem \ref{thm:VorticityDecay}
up to a factor $\snorm{x}^{-\varepsilon}$ for the purely periodic part.
Returning to the representation formula \eqref{eq:RepresentationVorticity_CurluWedgeu_pp},
we finally omit this factor and complete the proof of Theorem \ref{thm:VorticityDecay}.

To begin with, for $\radius\in[2\radius_0,\infty)$ consider the closed subset
\[
\calb_\radius\coloneqq\setcl{\zvel\in\vortspace}
{\norm{\zvel}_{\vortspace}\leq \CHsvelnorm+ \CHsvortnorm+1}
\]
of the Banach space $\vortspace$. 
Choose $\radius_1\in[2\radius_0,\infty)$ so large that 
for all $\radius\in[S_1,\infty)$ we have
\begin{align*}
\np{ \CFsvelnorm+ \CFsvortnorm}
\np{ \CHsvelnorm+ \CHsvortnorm+1}^2\radius^{-\varepsilon}
&\leq 1, 
\\
\np{ \CFsvelnorm+ \CFsvortnorm}
\np{ \CHsvelnorm+ \CHsvortnorm+1}\radius^{-\varepsilon} 
&\leq\fourth.
\end{align*}
Thus, we obtain the existence of a fixed point of $\zvel\mapsto\calf_\radius(\zvel)+\calh_\radius$.

\begin{cor}\label{cor:ExistenceFixedPoint}
For any $\radius\in[\radius_1,\infty)$ there is a function $\zvel_\radius\in\calb_\radius$ with
$\zvel_\radius=\calf_\radius(\zvel_\radius)+\calh_\radius$.
\end{cor}

\begin{proof}
By the Lemma \ref{lem:Hs_velest}, Lemma \ref{lem:Hs_vortest}, 
Lemma \ref{lem:Fs_velest} and Lemma \ref{lem:Fs_vortest}
and the choice of $\radius_1$, the mapping
\[
F_\radius\colon\calb_\radius\to\calb_\radius,\qquad
F_\radius(\zvel)\coloneqq\calf_\radius(\zvel)+\calh_\radius
\]
is a well-defined contractive self-mapping for any $\radius\geq\radius_1$. 
The contraction mapping principle thus implies 
the existence of the asserted fixed point $\zvel_\radius\in\calb_\radius$ 
of $F_\radius$.
\end{proof}

Next we show that $\zvel_\radius$ coincides with 
$\restriction{\uvel}{\torus\times\ball^\radius}$ 
for $\radius$ sufficiently large.
This yields pointwise estimates of $\uvel$.

\begin{lem}\label{lem:SlowerDecay}
There exists $\radius_2\in[\radius_1,\infty)$ such that for all $\radius\in[\radius_2,\infty)$ we have
\begin{align*}
\snorm{\curl\proj\uvel(x)}
&\leq\np{ \CHsvelnorm+ \CHsvortnorm+1}
\snorm{x}^{-3/2}\e^{-\frac{\wakefct{\Kconst x}}{1+\radius}}, \\
\snorm{\curl\projcompl\uvel(t,x)}
&\leq\np{ \CHsvelnorm+ \CHsvortnorm+1}
\snorm{x}^{-9/2+\varepsilon}\e^{-\frac{\wakefct{\Kconst x}}{1+\radius}}
\end{align*}
for all $t\in\torus$ and $x\in\ball^{\radius}$.
\end{lem}

\begin{proof}
For $\radius\geq2\radius_0$ we set $\Uvel_\radius\coloneqq\restriction{\uvel}{\torus\times\ball^\radius}$.
By Theorem \ref{thm:VelocityDecay} we know $\Uvel_\radius\in\velspace$ 
with $\velnorm{\Uvel}\leq \CVelocityDecay$, 
and by \eqref{eq:uFixedPoint_VortDecay} we have
$\Uvel_\radius=\calf_\radius(\Uvel_\radius)+\calh_\radius$ for any $\radius\geq 2\radius_0$. 
Now let $\radius\geq\radius_1$ and let $\zvel_\radius\in\calb_\radius$ be the function from 
Corollary \ref{cor:ExistenceFixedPoint}.
Then Lemma \ref{lem:Fs_velest} implies
\[
\begin{aligned}
\velnorm{\zvel_\radius-\Uvel_\radius}
=\velnorm{\calf_\radius(\zvel_\radius)-\calf_\radius(\Uvel_\radius)}
&\leq \CFsvelnorm\radius^{-\varepsilon}
\bp{\velnorm{\zvel_\radius}+\velnorm{\Uvel_\radius}}
\velnorm{\zvel_\radius-\Uvel_\radius}\\
&\leq \CFsvelnorm\radius^{-\varepsilon}
\bp{ \CHsvelnorm+ \CHsvortnorm+1+ \CVelocityDecay}
\velnorm{\zvel_\radius-\Uvel_\radius}.
\end{aligned}
\]
Choosing $\radius_2\in[\radius_1,\infty)$ such that for all $\radius\in[\radius_2,\infty)$ we have
\[
 \CFsvelnorm
\bp{ \CHsvelnorm+ \CHsvortnorm+1+ \CVelocityDecay}\radius^{-\varepsilon}
\leq\frac{1}{2},
\]
we conclude $\velnorm{\zvel_\radius-\Uvel_\radius}\leq\velnorm{\zvel_\radius-\Uvel_\radius}/2$
and hence $\velnorm{\zvel_\radius-\Uvel_\radius}=0$ for all $\radius\in[\radius_2,\infty)$.
This implies $\zvel_\radius=\Uvel_\radius=\restriction{\uvel}{\torus\times\ball^\radius}$.
In particular, we have
$\vortnorm{\restriction{\uvel}{\torus\times\ball^\radius}}=\vortnorm{\zvel_\radius}
\leq \CHsvelnorm+ \CHsvortnorm+1$ 
for all $\radius\in[\radius_2,\infty)$.
This completes the proof.
\end{proof}

Another application of the convolution formula \eqref{eq:RepresentationVorticity_CurluWedgeu_pp}
enables us to omit the term $\varepsilon$ in the estimate of $\curl\projcompl\uvel$,
which yields the estimates from Theorem \ref{thm:VorticityDecay} in the case $\Omega=\R^3$. 

\begin{thm}\label{thm:VorticityDecay_WholeSpace}
Let $\rey> 0$ and let $f\in\LR{q}(\torus\times\R^3)^3$ for all $q\in(1,\infty)$
have bounded support. 
Let $\uvel$ be a weak time-periodic solution to \eqref{sys:NavierStokesTP}
in the sense of Definition \ref{def:WeakSolution_NStp}, which satisfies
\eqref{el:VelocityAdditionalIntegrability}.
Then there exist constants $\CVorticityDecay>0$ and 
$\alpha=\alpha(\rey,\per)>0$ such that
the estimates
\eqref{est:VorticityDecay_ss} and 
\eqref{est:VorticityDecay_pp}
hold for all $(t,x)\in\torus\times\R^3$. 
\end{thm}

\begin{proof}
We decompose $\uvel=\vvel+\wvel$ into steady-state part $\vvel\coloneqq\proj\uvel$ 
and purely periodic part $\wvel\coloneqq\projcompl\uvel$. 
Since $\curl\uvel$ is bounded 
by Theorem \ref{thm:VelocityDecay}, 
Lemma \ref{lem:SlowerDecay} implies
\begin{equation}\label{est:VorticityDecay_BadEst}
\begin{aligned}
\snorm{\curl\vvel(x)}
&\leq C 
\np{1+\snorm{x}}^{-3/2}\e^{-\alpha\wakefct{x}}, \\
\snorm{\curl\wvel(t,x)}
&\leq C 
\np{1+\snorm{x}}^{-9/2+\varepsilon}\e^{-\alpha\wakefct{x}}
\end{aligned}
\end{equation}
for all $(t,x)\in\torus\times\R^3$, where $\alpha=\np{1+\radius_2}^{-1}\Kconst$.
In particular, this implies \eqref{est:VorticityDecay_ss},
and for \eqref{est:VorticityDecay_pp} 
it remains to remove $\varepsilon$ in the second inequality.
Due to Theorem \ref{thm:VelocityDecay}, 
the estimates \eqref{est:VorticityDecay_BadEst} further yield
\[
\snorml{
\curl\vvel(y)\wedge\wvel(s,y) 
+\curl\wvel(s,y)\wedge\vvel(y)
+\projcompl\nb{\curl\wvel\wedge\wvel}(s,y)
}
\leq  C \np{1+\snorm{y}}^{-9/2}\e^{-\alpha\wakefct{\rey y}}
\]
for all $\np{t,x}\in\torus\times\R^3$.
Moreover, by Theorem \ref{thm:FundsolVortpp_est} we have
\[
\int_\torus\snorm{\grad\fundsolvortpp(t-s,x-y)}\,\ds
\leq C \snorm{x-y}^{-5/2} \e^{- \Cexpfs\snorm{x-y}}.
\]
Using these estimates and \eqref{est:Conv_fLp_vortpp}
in the representation formula
\eqref{eq:RepresentationVorticity_CurluWedgeu_pp}, we conclude
\[
\snorm{\curl\wvel(t,x)}
\leq C\snorm{x}^{-9/2}\e^{-\frac{\Cexpfs\snorm{x}}{2\np{1+\radius_0}}}
+C \int_{\R^3}\snorm{x-y}^{-5/2} \e^{- \Cexpfs\snorm{x-y}}
\np{1+\snorm{y}}^{-9/2}\e^{-\alpha\wakefct{\rey y}}\,\dy.
\]
Due to $2\wakefct{\Kconst x}\leq \Cexpfs\snorm{x}$, we have 
\[
\half \Cexpfs\snorm{x-y}+\alpha\wakefct{\rey y}
\geq\wakefct{\Kconst\np{x-y}}+\frac{\wakefct{\Kconst y}}{1+\radius_2}
\geq\frac{\wakefct{\Kconst x}}{1+\radius_2}
=\alpha\wakefct{x},
\]
and we obtain
\[
\begin{aligned}
\snorm{\curl\wvel(t,x)}
&\leq 
C\np{1+\snorm{x}}^{-9/2}\e^{-\alpha\wakefct{x}}\\
&\qquad+C \e^{-\alpha\wakefct{x}}\int_{\R^3}\snorm{x-y}^{-5/2} 
\e^{- \Cexpfs\snorm{x-y}/2}\np{1+\snorm{y}}^{-9/2}\,\dy,
\end{aligned}
\]
where we used \eqref{est:Kconst}.
We estimate the remaining integral with Lemma \ref{lem:ConvolutionExp},
which leads to
\eqref{est:VorticityDecay_pp} 
and completes the proof.
\end{proof}

Finally, we employ a classical cut-off argument
to extend the result to an exterior domain and to
finish the proof of Theorem \ref{thm:VorticityDecay}.

\begin{proof}[Proof of Theorem \ref{thm:VorticityDecay}]
First of all, Lemma \ref{lem:NavierStokesTP_Regularity}
implies the existence of a pressure field $\upres$
such that $\np{\uvel,\upres}$ is a strong solution 
to \eqref{sys:NavierStokesTP_ExteriorDomain}
satisfying 
\eqref{el:NavierStokesTP_Regularity_ss}--\eqref{el:NavierStokesTP_Regularity_pressure}.
Fix radii $R>r>0$ such that $\partial\Omega\subset\ball_r$,
and let $\cutoff\in\CRi(\R^3)$ be a cut-off function such that
$\cutoff(x)=0$ for $\snorm{x}\leq r$ and 
$\cutoff(x)=1$ for $\snorm{x}\geq R$.
By the divergence theorem, $\Div\uvel=0$ and \eqref{cond:BoundaryData},
we have
\[
\begin{aligned}
\int_{\ball_{r,R}}\uvel\cdot\grad\cutoff\,\dx
&=\int_{\ball_R} \Div\bp{\uvel\np{\cutoff-1}}\,\dx
=-\int_{\partial\Omega}\uvel\cdot\nvec\,\dS
=0.
\end{aligned}
\]
Therefore, there exists a function $\Vvel$ with
$\Vvel\in\WSR{1,2}{q}(\torus\times\R^3)^3$ for all $q\in(1,\infty)$ 
and $\supp\Vvel\subset\torus\times\ball_{r,R}$
such that
$\Div\Vvel=\uvel\cdot\grad\cutoff$;
see \cite[Section III.3]{GaldiBookNew} for example.
We define $\Uvel\coloneqq\cutoff\uvel-\Vvel$ and
$\Upres\coloneqq\cutoff\upres$.
Then $\Uvel\in\LR{r}(\torus\times\R^3)$ for some $r\in(5,\infty)$
and $\Uvel$ is a weak solution to
\begin{equation}\label{sys:NavierStokesTP_CutOff}
\begin{pdeq}
\partial_t\Uvel-\Delta\Uvel-\rey\partial_1\Uvel+\Uvel\cdot\grad\Uvel+\grad\Upres&=F
&&\tin\torus\times\R^3, \\
\Div\Uvel&=0 
&&\tin\torus\times\R^3, \\
\lim_{\snorm{x}\to\infty}\Uvel(t,x)&=0 
&&\tfor t\in\torus,
\end{pdeq}
\end{equation}
in the sense of Definition \ref{def:WeakSolution_NStp},
where $F\in\LR{q}(\torus\times\R^3)^3$ for all $q\in(1,\infty)$,
and $\supp F$ is compact. 
Now the assertion follows from Theorem \ref{thm:VorticityDecay_WholeSpace}
and the identity $\Uvel(t,x)=\uvel(t,x)$ for $\snorm{x}>R$.
\end{proof}


\smallskip\par\noindent
Weierstrass Institute for Applied Analysis and Stochastics\\ 
Mohrenstra\ss{}e 39, 10117 Berlin, Germany\\
Email: {\texttt{eiter@wias-berlin.de}}
\medskip\smallskip\par\noindent
Department of Mechanical Engineering and Materials Science\\
University of Pittsburgh\\
Pittsburgh, PA 15261, USA\\
Email: {\texttt{galdi@pitt.edu}}

\end{document}